\documentclass[alpha-refs,serif]{wiley-article}

\usepackage{siunitx}
\RequirePackage{hypernat}

\def\hat{\widehat}
\def\tilde{\widetilde}

\newtheorem{assumption}{Assumption}

\newcommand{\diag}{\mbox{diag}}

\newcommand{\yu}{{}}

\papertype{Original Article}
\paperfield{Journal Section}

\title{High Dimensional Robust Inference for Cox Regression Models}

\abbrevs{High Dimensional Robust Inference}

\author[1]{Shengchun Kong}
\author[2]{Zhuqing Yu}
\author[3]{Xianyang Zhang}
\author[4]{Guang Cheng}

\affil[1]{Manager, Biostatistics, Gilead Sciences Inc. 333 Lakeside Dr., Foster city, CA 94404. Email: kongshengchun@gmail.com}
\affil[2]{Graduate Student.Department of Statistics, Purdue University. 250 N. University St., West Lafayette, IN 47907. Email: zhuqing.yu.stat@gmail.com. Currently employed at AbbVie, Inc. North Chicago, IL 60064.}
\affil[3]{Assistant Professor. Department of Statistics, Texas A\&M University, College Station, Texas 77843. Email: {zhangxiany@stat.tamu.edu}}
\affil[4]{Professor. Department of Statistics, Purdue University. 250 N. University St., West Lafayette, IN 47907. Email: chengg@purdue.edu}

\corraddress{Guang Cheng. Department of Statistics, Purdue University. 250 N. University St., West Lafayette, IN 47907}
\corremail{chengg@purdue.edu}

\fundinginfo{Zhuqing Yu is partially Supported by NSF DMS-1151692, DMS-1418042. Guang Cheng's Research is sponsored by NSF CAREER Award DMS-1151692, DMS-1418042, DMS-1712907, DMS-1811812, DMS-1821183, Simons Fellowship in Mathematics, Office of Naval Research (ONR N00014-15-1-2331, ONR N00014-18-2759) and a grant from Indiana Clinical and Translational Sciences Institute.}

\runningauthor{Shengchun Kong et al.}

\begin{document}

\maketitle

\begin{abstract}
We consider high-dimensional inference for potentially
misspecified Cox proportional hazard models based on low dimensional results by \cite{lin:1989}. A de-sparsified Lasso estimator is proposed based on the log partial likelihood function
and shown to converge to a pseudo-true parameter vector. Interestingly, the sparsity of the true parameter can be inferred from that of the above limiting parameter. Moreover, each component of the above (non-sparse) estimator is shown to be asymptotically normal with a variance that can be consistently estimated even under model misspecifications. In some cases, this asymptotic distribution leads to valid statistical inference procedures, whose empirical performances are illustrated through numerical examples.

\keywords{Cox regression, debiased Lasso, high dimension, partial likelihood, robust inference}
\end{abstract}

\section{Introduction}
With rapid technology advances, it is now possible to collect a
large amount of healthcare data as often observed in genomic or
image studies. In survival analysis, data are often analyzed to
investigate how covariates for patient information affect the
occurrence of some events such as disease. The Cox proportional
hazards model \citep{cox:1972} is one of the most widely used
survival models for censored time-to-event data. When the number of
covariates collected is larger than sample size, high dimensional
regularized Cox regression (e.g., under Lasso penalty) has been
proposed in the literature, e.g., \cite{gui:2005, bradic:2011,
jelena:2015}. In particular, \cite{kong:2014} and \cite{huang:2014}
studied finite sample oracle inequalities for Lasso regularized Cox
models under random and fixed designs, respectively.

As a natural followup work, we consider {\em high dimensional
inference} for Cox regression models under possible
misspecifications. Recent process towards high dimensional
inference is mostly concerned with (generalized) linear models. One particular method is
through de-sparsifying Lasso estimator; see \cite{van:2014},
\cite{javanmard:2014} and \cite{zhang:2014}. In this paper, a
similar de-sparsified Lasso estimator is proposed based on the log
partial likelihood of Cox regression models. A key technical condition in justifying high dimension inference
for (generalized) linear models is that the summands in the log
likelihood need to be Lipschitz and independently and identically distributed (i.i.d.).
Unfortunately, the summands in the log partial likelihood for censored survival data are neither i.i.d. nor Lipschitz. One major novelty in our theoretical analysis is to introduce an
intermediate function with a sum of i.i.d. non-Lipschitz terms for
approximating the above log partial likelihood as in \cite{kong:2014}.
We further apply mean value theorem to deal with the non-Lipschitz loss function under a bounded condition, i.e., Assumption~\ref{assp:xbeta}.

This novel technical device enables us to derive the limiting value of our
proposed (non-sparse) estimator, called as {\em pseudo-true} parameter, which turns out to be determined by the intermediate loss function proposed above and can be interpreted meaningfully; see Example 1. Note that the pseudo-true parameter is not necessarily
sparse even if the underlying true hazard function depends only on a few covariates. Fortunately, we are able to identify a situation where the sparsity of the true parameter can be inferred from that of the above limiting parameter. Specifically, the inactive variables are always included in the sparsity set estimated through the working model.

Another crucial feature of our work is that it does not require the
model to be correctly specified. The consequences of misspecifying
low dimensional Cox models have been extensively investigated in
\cite{gail:1984, struthers:1986, lin:1989} among others. High dimensional inference for
misspecified linear models have recently been studied in
\cite{buhlmann:2015}. To perform valid statistical inference, we establish the asymptotic distribution that centers around the pseudo-true parameter, and further provide a robust variance estimation formula that works even under model misspecifications. Empirical performances are demonstrated in both correctly specified and misspecified Cox regression models. While this manuscript was under preparation, we note an arxiv work \citep{fang:2014} for high dimensional inference on {\em correctly specified} Cox regression based on decorrelated method. \yu{ During our revision, we also note another arxiv work \citep{yu:2018} for constructing confidence intervals for high dimensional Cox model based on CLIME estimator \citep{cai:2011}, where the covariates are possibly time-dependent.} Nevertheless, our inference results are constructed based on de-sparsified Lasso with a particular focus on {\em misspecification}, and analyzed through a different intermediate function approach.

\section{Robust de-sparsified Lasso estimator}
Consider a survival model with a true hazard function
$\lambda_0(t|\bm{X})$ for a failure time $T$ given a covariate vector
${\bm X}=(X_1,\dots,X_p)^T$. Denote $C$ as the censoring time,
$Y=\min(T,C)$ and  $\Delta=1(T\le C)$. Let $(Y_i, \Delta_i,
{\bm X_i})_{i=1}^n$ be $n$ i.i.d. observations from the underlying true
model and $\bf{X}=$ $(\bm X_1^T,\ldots,\bm X_n^T)^T$ be the $n\times p$ design matrix.
We fit a potentially misspecified {\em working model} to the
above observations:
\begin{eqnarray}\label{cox1}
\lambda(t|{\bm X})=\lambda(t)\exp({\bm X}^T\beta),
\end{eqnarray}
where $\beta=(\beta_1,\dots,\beta_p)^T$ and $\lambda(t)$ is an
unknown baseline hazard function. Note that $\lambda_0(t|{\bm X})$ does
not need to take an exponential regression form, or has to be a
proportional hazard model.

Under the working model (\ref{cox1}), the negative log partial likelihood function is written as
\begin{equation}
l_n(\beta)=-\frac{1}{n} \sum_{i=1}^n \left[{\bm X_i}^T\beta-\log\left\{\frac{1}{n} \sum_{j=1}^n 1(Y_j\geq Y_i)\exp({\bm X_j}^T\beta)\right\}\right]\Delta_i, \label{partiallikelihood}
\end{equation}
with its first and second derivatives
\begin{align*}
\dot{l}_n(\beta)=-\frac{1}{n}\sum_{i=1}^n \left\{{\bm X_i}-\frac{\widehat{\mu}_1(Y_i;\beta)}{\widehat{\mu}_0(Y_i;\beta)}\right\}\Delta_i,\quad\ddot{l}_n(\beta)=\frac{1}{n}\sum_{i=1}^n \left\{\frac{\widehat{\mu}_2(Y_i;\beta)}{\widehat{\mu}_0(Y_i;\beta)}
-\left[\frac{\widehat{\mu}_1(Y_i;\beta)}{\widehat{\mu}_0(Y_i;\beta)}\right]^{\otimes 2}\right\}\Delta_i,
\end{align*}
where $\widehat{\mu}_r(t;\beta)=n^{-1}\sum_{j=1}^n 1(Y_j\geq t){\bm X_j}^{\otimes r}\exp({{\bm X_j}^T\beta})$ and \yu{${\otimes}$ represents the Kronecker product.\footnote{For a matrix $A\in \mathbb{R}^{m\times n}$ and a matrix $B \in \mathbb{R}^{p\times q}$, the Kronecker product $A\otimes B$ is a matrix in $\mathbb{R}^{mp\times nq}$ such that \[
\begin{bmatrix}
    a_{11}B & \dots  & a_{1n}B \\
    \vdots  & \ddots &\vdots \\
    a_{m1}B & \dots  & a_{mn}B
\end{bmatrix}
\].
} In particular, we have {${\bm X_j}^{\otimes 0}=1, {\bm X_j}^{\otimes 1}={\bm X_j}$ and ${\bm X_j}^{\otimes 2}={\bm X_j}{\bm X_j}^T$.}}

The Lasso estimator for $\beta$ is defined as
\begin{equation*}
\widehat{\beta}:=\arg\min_{\beta\in\mathcal{R}^p} \left\{l_n(\beta)+2\lambda \|\beta\|_1\right\},\label{def:Lasso}
\end{equation*}
where $\|\cdot\|_1$ is the $\ell_1$ norm. It is known that $\widehat{\beta}$ does not possess a tractable limiting distribution \citep{kong:2014,huang:2014}. Inspired by the recent de-sparsifying idea, we construct a non-sparse estimator by inverting the Karush-Kuhn-Tucker (KKT) condition:
\begin{equation*}\label{bhat}
\widehat{b}:=(\widehat{b}_1,\dots,\widehat{b}_p)^T=\widehat{\beta}-\widehat{\Theta}\dot{l}_n(\widehat{\beta}),
\end{equation*}
where $\widehat \Theta$ is a reasonable approximation for the inverse of $\widehat{\Sigma}:=\ddot{l}_n(\widehat{\beta})$. We remark that the procedure of constructing $\widehat{b}$ remains the same regardless whether the working model (\ref{cox1}) is correctly specified or not. As will be shown in Section~\ref{sec:the}, the limiting value of $\widehat b$ can be interpreted meaningfully. Moreover, $\widehat b$ is shown to be asymptotically normal, whose variance can be estimated consistently even under model misspecifications.

The approximation $\widehat \Theta$ can be constructed by performing nodewise Lasso as follows. We first re-write $\ddot{l}_n(\beta)$ as a product of a matrix and its transpose:
\begin{eqnarray}\label{prodfor}
\ddot{l}_n(\beta)= C_{\beta}^T C_{\beta},
\end{eqnarray}
where
\begin{eqnarray}\label{cmatrix}
&&C_{\beta}:=\begin{pmatrix}
             n^{-1}\Delta_1 1(Y_1\geq Y_1)\sqrt{\frac{\exp({\bm X_1}^T{\beta})}{\widehat{\mu}_0(Y_1;{\beta})}}\left({\bm X_1}^T-\frac{\widehat{\mu}_1^T(Y_1;{\beta})}{\widehat{\mu}_0(Y_1;{\beta})}\right) \\
             \vdots \\
             n^{-1}\Delta_1 1(Y_n\geq Y_1)\sqrt{\frac{\exp({\bm X_n}^T{\beta})}{\widehat{\mu}_0(Y_1;{\beta})}}\left({\bm X_n}^T-\frac{\widehat{\mu}_1^T(Y_1;{\beta})}{\widehat{\mu}_0(Y_1;{\beta})}\right) \\
              \vdots \\
             n^{-1}\Delta_n 1(Y_1\geq Y_n)\sqrt{\frac{\exp({\bm X_1}^T{\beta})}{\widehat{\mu}_0(Y_n;{\beta})}}\left({\bm X_1}^T-\frac{\widehat{\mu}_1^T(Y_n{\beta})}{\widehat{\mu}_0(Y_n;{\beta})}\right) \\
              \vdots \\
              n^{-1}\Delta_n 1(Y_n\geq Y_n)\sqrt{\frac{\exp({\bm X_n}^T{\beta})}{\widehat{\mu}_0(Y_n;{\beta})}}\left({\bm X_n}^T-\frac{\widehat{\mu}_1^T(Y_n;{\beta})}{\widehat{\mu}_0(Y_n;{\beta})}\right) \\
            \end{pmatrix}_{n^2\times p}.
\end{eqnarray}
Denote $C_{\widehat\beta,j}$ as the $j$-th column of $C_{\widehat{\beta}}$ and $C_{\widehat\beta,-j}$ as the submatrix of $C_{\widehat{\beta}}$ without $C_{\widehat\beta,j}$. Based on the decomposition (\ref{prodfor}), we run the following nodewise Lasso $p$ times:
\begin{equation}\label{nodewiseLasso}
\widehat{\gamma}_j:=\arg\min_{\gamma}\{\|C_{\widehat{\beta},j}-C_{\widehat{\beta},-j}\gamma\|^2+2\lambda_j\|\gamma\|_1\},
\end{equation}
where $\widehat{\gamma}_j=\{\widehat{\gamma}_{j,k}; k=1,\dots,p, k\ne j\}$. Define
$\widehat{\tau}_j^2=\widehat{\Sigma}_{j,j}-\widehat{\Sigma}_{j,\backslash j} \widehat{\gamma}_j, $
where $\widehat{\Sigma}_{j,j}$ denotes the $j$-th diagonal element
of $\widehat{\Sigma}$ and $\widehat{\Sigma}_{j,\backslash j}$
denotes the $j$-th row of $\widehat{\Sigma}$ without
$\widehat{\Sigma}_{j,j}$. We now define
\begin{equation*}\label{Thetahat}
\widehat{\Theta}:=\widehat{T}^{-2}\widehat{G},
\end{equation*}
where
\begin{eqnarray*}\label{chat}
\widehat{T}^2:=\diag(\widehat{\tau}_1^2,\cdots,\widehat{\tau}_p^2)\;\;\;\;\mbox{and}\;\;\;\;
\widehat{G}:=\begin{pmatrix}
1 & -\widehat{\gamma}_{1,2} & \cdots & -\widehat{\gamma}_{1,p} \\
-\widehat{\gamma}_{2,1} & 1 & \cdots & -\widehat{\gamma}_{2,p} \\
\vdots & \vdots & \ddots & \vdots \\
-\widehat{\gamma}_{p,1} & -\widehat{\gamma}_{p,2} & \cdots & 1 \\
\end{pmatrix}.
\end{eqnarray*}

\section{Theoretical properties}\label{sec:the}
\subsection{Pseudo-true Parameter}
In this section, we derive the limiting value of $\widehat{b}$, denoted as $\beta_0$, and further discuss the meaning of its sparsity. As discussed previously, the summands in log partial likelihood (\ref{partiallikelihood}), based on which $\widehat b$ is constructed, are neither i.i.d. nor Lipschitz. Therefore, we first need to introduce an intermediate function that approximates (\ref{partiallikelihood}):
\begin{equation*}
\widetilde{l}_n(\beta):= -\frac{1}{n}\sum_{i=1}^n
\left\{{\bm X_i}^T\beta-\log \mu_0(Y_i;\beta)\right\}\Delta_i,
\end{equation*}
where $\mu_r(t;\beta)=\mathbb{E}\{\widehat{\mu}_r(t;\beta)\}=\mathbb{E}\left\{1(Y\geq t){\bm X}^{\otimes r}\exp({\bm X}^T\beta)\right\}.$
As implied by Theorem \ref{thm:normalthm}, the {\em pseudo-true} parameter $\beta_0$ is the unique solution to a system of $p$ equations,
\begin{eqnarray}\label{pseudo}
\int_0^{\infty}\mu_1(t)dt-\int_0^{\infty}\frac{\mu_1(t;\beta)}{\mu_0(t;\beta)} \mu_0(t)dt=0,
\end{eqnarray}
where $\mu_r(t)=\mathbb{E}\{1(Y\geq t){\bm X}^{\otimes r}\lambda_0(t|{\bm X})\}$, provided that
\begin{equation}{\label{variance}}
\Sigma_{\beta_0}={\displaystyle\int_0^{\infty}}\left\{\frac{\mu_2(t;\beta_0)}{\mu_0(t;\beta_0)}
-\left[\frac{\mu_1(t;\beta_0)}{\mu_0(t;\beta_0)}\right]^{\otimes 2}\right\}\mu_0(t)dt
\end{equation}
is positive definite. It is easy to verify that  (\ref{pseudo}) and (\ref{variance}) turn out to be $\mathbb{E}\dot{\widetilde{l}}_n(\beta_0)=0$ and  $\mathbb{E}\ddot{\widetilde{l}}_n(\beta_0)=\Sigma_{\beta_0}$, respectively. Hence, $\tilde{l}_n$ indeed plays a similar role as a true likelihood for $\beta$.

From the example below, we note that $\widehat b$ with some particular limit value $\beta_0$ can still be useful for some statistical inference problems.
\begin{example}[Example 1]\label{example}
Suppose that the true hazard function is $\lambda_1(t)\exp(\gamma X_1^2)$ in comparison with the working model $\lambda(t)\exp({\bm X}^T\beta)$. Let $X_{-1}$ be the sub-vector of ${\bm X}$ without the first element $X_1$. If we assume that $ X_{-1}$ is independent of $ X_1$ and is symmetric about zero, it can be shown by substituing into (\ref{pseudo}) that $\beta_{0}=(0,\ldots,0)^T$, provided that the censoring time $C$ is independent of ${\bm X}$. In this case, according to Theorem \ref{thm:normalthm} below, we can construct a valid test based on $\widehat b_j$ for testing the null hypothesis that the failure time does not depend on $X_j$, for any $1\le j\le p$.
\end{example}


The pseudo-true parameter $\beta_0$ defined in (\ref{pseudo}) is not necessarily sparse even if the underlying true hazard function only depends on a few covariates. 
Theorem~\ref{thm:projection} says that if we infer a variable as an active variable (significantly different from zero) in the working model, it must be an active variable in the true model. Interestingly, this directly implies $\beta_{0j}=0$ for $2\le j\le p$  in Example 1 without doing any calculation. 

Define $S_0=\{j: \beta_{0j} \neq 0\}$ and $S_{\lambda_0}$ $(S_{C_0})$ as the index set of all variables having an influence on the true conditional hazard function $\lambda_0(t|\cdot)$ (conditional distribution of $C$ given ${\bm X}$). Let $\bm X_1^*$ $({\bm X}_2^*)$ be a sub-vector of ${\bm X}$ with $S_{\lambda_0}\cup S_{C_0}$ (the complement of $S_{\lambda_0}\cup S_{C_0})$ being its index set.

\begin{theorem}\label{thm:projection}
Suppose that $\mathbb{E}({\bm X}_2^*\mid \bm X_1^*)=0$. Then we have $S_0\subseteq S_{\lambda_0}\cup S_{C_0}.$ If we further assume that the censoring time $C$ is independent of ${\bm X}$, then $S_0\subseteq S_{\lambda_0}$.
\end{theorem}

In the theorem above, we do not need Gaussian design condition, which is required in \cite{buhlmann:2015} for misspecified linear models. Rather, a conditional expectation condition $\mathbb{E}({\bm X}_2^*\mid \bm X_1^*)=0$ suffices (even for generalized linear regression).

\subsection{Asymptotic Distribution}{\label{sec:asymp}}


In this section, we show that $n^{1/2}(\widehat{b}-\beta_0)$ converges to a normal distribution and further provide a {\em robust} variance estimate formula that is consistent even under misspecifications.

Recall that $\widetilde l_n$ is the intermediate function. Some straightforward calculation shows that $\ddot{\widetilde{l}}_n(\beta)$ can be re-written as (in comparison with (\ref{prodfor})) $$\ddot{\widetilde{l}}_n(\beta)=\frac{1}{n}\sum_{i=1}^n \left\{\frac{\mu_2(Y_i;\beta)}{\mu_0(Y_i;\beta)}
-\left[\frac{\mu_1(Y_i;\beta)}{\mu_0(Y_i;\beta)}\right]^{\otimes 2}\right\}\Delta_i=\mathbb{E}_{X,Y}(D^T_{\beta}D_{\beta}),$$
where $\mathbb{E}_{{\bm X},Y}$ denotes the expectation with respect to ${\bm X}$ and $Y$ only, and
\begin{eqnarray*}
    &&D_{\beta}:=\begin{pmatrix}
        \frac{1}{n}\Delta_1 1(Y\geq Y_1)\sqrt{\frac{\exp({\bm X}^T\beta)}{\mu_0(Y_1;\beta)}}\left({\bm X}^T-\frac{\mu_1^T(Y_1;\beta)}{\mu_0(Y_1;\beta)}\right) \\
        \vdots \\
        \frac{1}{n}\Delta_1 1(Y\geq Y_1)\sqrt{\frac{\exp({\bm X}^T\beta)}{\mu_0(Y_1;\beta)}}\left({\bm X}^T-\frac{\mu_1^T(Y_1;\beta)}{\mu_0(Y_1;\beta)}\right) \\
        \vdots \\
        \frac{1}{n}\Delta_n 1(Y\geq Y_n)\sqrt{\frac{\exp({\bm X}^T\beta)}{\mu_0(Y_n;\beta)}}\left({\bm X}^T-\frac{\mu_1^T(Y_n;\beta)}{\mu_0(Y_n;\beta)}\right) \\
        \vdots \\
        \frac{1}{n}\Delta_n 1(Y\geq Y_n)\sqrt{\frac{\exp({\bm X}^T\beta)}{\mu_0(Y_n;\beta)}}\left({\bm X}^T-\frac{\mu_1^T(Y_n;\beta)}{\mu_0(Y_n;\beta)}\right) \\
    \end{pmatrix}_{n^2\times p}.
\end{eqnarray*}


Before stating our main assumptions, we need the following notation. For $1\le j\le p$, define $\gamma_{\beta_0,j}= \arg\min_{\gamma}\mathbb{E} (D_{\beta_0,j}-D_{\beta_0,-j}\gamma)^2$.
Let $\Theta_{\beta_0}=\Sigma_{\beta_0}^{-1}$ (assume to exist) and $\tau_{\beta_0,j}^2:=\Theta_{\beta_0,jj}^{-1}$. For simplicity, we write $\Theta=\Theta_{\beta_0}$. Recall that $\bf{X}=$$(\bm X_1^T,\ldots,\bm X_n^T)^T$ is the $n\times p$ design matrix.

\begin{assumption}\label{assp:xbounded}
	$\|\bf{X}\|_{\infty}=\max_{i,j}|{\bm X}_{i,j}|\leq K_1<\infty$.
\end{assumption}
\begin{assumption}\label{assp:xbeta}
 $\|\bf{X}\beta_0\|_{\infty}\leq K_2<\infty$.
\end{assumption}
\begin{assumption}\label{assp:agamma}
	$\|A_{\beta_0,-j}\gamma_{\beta_0,j}\|_{\infty}=\mathcal{O}(1)$ $\forall\ j$, where $A_{\beta_0,-j}$ is defined in the Appendix.
\end{assumption}
\begin{assumption}\label{assp:eigen}
	 The smallest eigenvalue of $\Sigma_{\beta_0}$ is bounded away from zero and  $\|\Sigma_{\beta_0}\|_{\infty}=\mathcal{O}(1)$.
\end{assumption}
\begin{assumption}\label{assp:stoptime}
	The observation time stops at a finite time $\tau>0$ with probability $\xi:=P(Y\geq \tau)>0$.
\end{assumption}
\begin{assumption}\label{assp:s0sj}
	 $s_0=o(\sqrt{n}/\log(p))$ and $s_j=o(\{\sqrt{n}/\log(p)\}^{2/3})$, where $s_0=|S_0|$ and $s_j$ is the number of off-diagonal non-zeros of the $j$-th row of $\Theta$.
\end{assumption}
\begin{assumption}\label{assp:lambda0j}
$\lambda\asymp \sqrt{\log p/n}$ and $\lambda_j\asymp \sqrt{\log p/n}$ uniformly for $1\le j\le p.$
\end{assumption}

Assumptions \ref{assp:xbounded}\ -- \ref{assp:eigen} are the same as Conditions (iii), (iv) and (v) in Theorem 3.3 of \cite{van:2014}. Assumption \ref{assp:stoptime} is typically required in survival analysis, see \cite{andersen:1982}. The condition on $s_0$ in
Assumption \ref{assp:s0sj} is also typical for the de-sparsified Lasso method, while the condition imposed on $s_j$ is to ensure that the $\ell_1$ difference between $\widehat{\Theta}_j$ and $\Theta_j$ is of the order $o_P(1/\sqrt{\log p})$, where $\widehat{\Theta}_j$ and $\Theta_j$ are the $j$-th rows of $\widehat{\Theta}$ and $\Theta$ respectively. Note that Assumption \ref{assp:xbeta} significantly relaxes the bounded condition on $\sup_{\beta} \|\beta\|_1$ imposed in \cite{kong:2014}. In fact, with Assumption \ref{assp:xbeta}, we can obtain a similar non-asymptotic oracle inequality as that in \cite{kong:2014} by choosing a slightly larger constant in the tuning parameter $\bar{\lambda}_{n,0}^B$ defined therein.


Lemma \ref{lemma:betaconsis} (Lemma \ref{lem:tj1}) describes the difference between $\widehat{\beta}$ and $\beta_0$ ($\widehat{\Theta}_j$ and $\Theta_j$). We omit the proof of Lemma~\ref{lemma:betaconsis}, which can be straightforwardly adapted from \cite{kong:2014} under a weaker condition Assumption \ref{assp:xbeta} as discussed above. Our proof of Lemma~\ref{lem:tj1} differs from \cite{van:2014} as $\gamma_{\beta_0,j}$ (used in the analysis of $\widehat{\tau}_{\widehat{\beta},j}^2-\tau_{\beta_0,j}^2$) does not necessarily minimize $\mathbb{E}(C_{\beta_0,j}-C_{\beta_0,-j}\gamma)^2$. This is due to the introduction of our intermediate function.

\begin{lemma}\label{lemma:betaconsis}
    Under Assumptions \ref{assp:xbounded} - \ref{assp:lambda0j}, we have
    \begin{eqnarray*}
\|\widehat{\beta}-\beta_0\|_1=\mathcal{O}_P(\lambda s_0),\qquad \|\bf{X}(\widehat{\beta}-\beta_0)\|^2/n=\mathcal{O}_P(\lambda^2 s_0).\label{betaconsis3}
    \end{eqnarray*}

\end{lemma}
\begin{lemma}\label{lem:tj1}
    Under Assumptions \ref{assp:xbounded} - \ref{assp:lambda0j}, we have for every $1\le j\le p$,

    \begin{eqnarray*}
        &&\|\widehat{\Theta}_j-\Theta_j\|_1=\mathcal{O}_P(\lambda s_j^{3/2}\vee \lambda \sqrt{s_0s_j}),\qquad \|\widehat{\Theta}_j-\Theta_j\|_2=\mathcal{O}_P( \lambda \sqrt{s_0} \vee \lambda s_j),
    \end{eqnarray*}
    and
        \begin{equation*}
        |\widehat{\tau}_{\widehat{\beta},j}^2-\tau_{\beta_0,j}^2|=\mathcal{O}_P(s_j\sqrt{\log p/n}\vee\lambda \sqrt{s_0}).
        \end{equation*}
    Moreover,
    \begin{eqnarray*}
        &&|\widehat{\Theta}_j^T\Sigma_{\beta_0}\widehat{\Theta}_j-\Theta_{jj}|\leq (\|\Sigma_{\beta_0}\|_{\infty}\|\widehat{\Theta}_j-\Theta_j\|_1^2)\wedge (\Lambda_{\max}^2\|\widehat{\Theta}_j-\Theta_j\|_2^2)+2|\widehat{\tau}_{\widehat{\beta},j}^2-\tau_{\beta_0,j}^2|,
    \end{eqnarray*}
where $\Lambda_{\max}^2$ is the largest eigenvalue of $\Sigma_{\beta_0}$.
\end{lemma}

Lemma~\ref{lemma:scorenormal} shows the asymptotic normality of the score statistic $\dot{l}_n(\beta_0)$ under high dimensional setting, which is similar to \cite{lin:1989} for any fixed $p$.

\begin{lemma}\label{lemma:scorenormal}
      Under Assumptions \ref{assp:xbounded} - \ref{assp:lambda0j}, we have
    $$
    \frac{\sqrt{n}\widehat{\Theta}_j^T \dot{l}_n(\beta_0)}{\sqrt{\widehat{\Theta}_j^T \mathbb{E}\left\{n^{-1}\sum_{i=1}^n v_i(\beta_0)^{\otimes 2}\right\}\widehat{\Theta}_j}}
    $$
    converges weakly to $\mathcal{N}(0,1)$, where
    $$
    v_i(\beta)=\Delta_i\bigg\{{\bm X_i}-\frac{\mu_1(Y_i;\beta)}{\mu_0(Y_i;\beta)}\bigg\}-\int_0^{\infty}\frac{1(Y_i\geq t)\exp\{{\bm X_i}^T\beta\}}{\mu_0(t;\beta)}\bigg\{{\bm X_i}-\frac{\mu_1(Y_i;\beta)}{\mu_0(Y_i;\beta)}\bigg\}d\widetilde{F}(t),
    $$
    and
    $\widetilde{F}(t)=\mathbb{E} 1(Y \leq t, \Delta=1)$.
\end{lemma}

From Lemmas \ref{lemma:betaconsis} -- \ref{lemma:scorenormal}, we obtain our main results on the asymptotic normality of $\widehat b_j$. In particular, the asymptotic variance formula (\ref{var:mis}) in Theorem~\ref{thm:normalthm} (also used in \cite{lin:1989}  for low dimensional case) is robust in the sense that it can be applied irrespective whether the model is correct or not, while (\ref{var:correct}) in Corollary~\ref{cor:nomis} only holds for correctly specified models.

\begin{theorem}{\label{thm:normalthm}}
     Under Assumptions \ref{assp:xbounded} - \ref{assp:lambda0j}, we have for every $1\le j\le p$,
    \begin{equation*}
    \sqrt{n}\left(\widehat{b}_j-\beta_{0j}\right)/\widehat{\sigma}_j=V_j+o_P(1),
    \end{equation*}
    where $V_j$ converges weakly to $\mathcal{N}(0,1)$ and
    \begin{equation}\label{var:mis}
    \widehat{\sigma}_j^2=\widehat{\Theta}_j^T\left\{n^{-1}\sum_{i=1}^n \widehat{v}_i(\beta_0)^{\otimes 2}\right\} \widehat{\Theta}_j,
    \end{equation}
    with $$\widehat{v}_i(\beta)=\Delta_i\bigg\{{\bm X_i}-\frac{\widehat{\mu}_1(Y_i;\beta)}{\widehat{\mu}_0(Y_i;\beta)}\bigg\}-\sum_{k=1}^n\frac{\Delta_k 1(Y_i\geq Y_k)\exp\{{\bm X_i}^T \beta\}}{n\widehat{\mu}_0(Y_k; \beta)}\bigg\{{\bm X_i}-\frac{\widehat{\mu}_1(Y_k;\beta)}{\widehat{\mu}_0(Y_k;\beta)}\bigg\}.$$
\end{theorem}
\begin{corollary}\label{cor:nomis}
If the working model (\ref{cox1}) is correctly specified, we have for every $1\le j\le p$,
\begin{equation*}
\sqrt{n}\left(\widehat{b}_j-\beta_{0j}\right)/\widetilde{\sigma}_j=V_j+o_P(1),
\end{equation*}
where $V_j$ converges weakly to $\mathcal{N}(0,1)$ and
\begin{equation}\label{var:correct}
 \widetilde{\sigma}_j^2=\widehat{\Theta}_j^T\ddot{l}_n(\widehat{\beta})\widehat{\Theta}_j.
\end{equation}

\end{corollary}

\section{Numerical study}
We conducted extensive simulations to investigate the finite sample performances of our high dimensional inference methods. The rows of $\bf{X}$ were drawn independently from $\mathcal{N}(0,\Sigma)$ with each element truncated by $[-3,3]$. Constant censoring time was generated to yield $15\%$ and $30\%$ censoring rates. The Lasso estimator $\widehat{\beta}$ was obtained with a tuning parameter $\lambda$ from $10$-fold cross-validation, while $\lambda_j$'s in nodewise Lasso were also chosen by $10$-fold cross-validation. We set $(n,p)=(100, 500)$ \yu{and $(n,p)=(200, 30)$ with $1000$ replications. Note that when $(n,p)=(200, 30)$, we compare our results with those derived from partial likelihood estimation method. All the simulations were done on Purdue University rice cluster. For the case $(n,p)=(100, 500)$, it took approximately 4 hours for $1000$ replications run on one node with two 10-core Intel Xeon-E5 processors (20 cores per node) and 64 GB of memory. For $(n,p)=(200,30)$, it took approximately 1 hour for $1000$ replications based on de-sparsified Lasso method. For the real example in section~\ref{sec:real}, it took approximately 6 hours. }

\subsection{Correctly specified Cox regression model}
Assume $\lambda(t|{\bm X})=\exp({\bm X}^T\beta)$ is the true hazard function with two different 
covariance matrices $\Sigma$:
\begin{itemize}
    \item Independent: $\Sigma=I$ ,
    \item Equal Correlation: $\Sigma=\widetilde \Sigma_p$.\footnote{Here $\widetilde{\Sigma}_j$ for $1\le j\le p$ denotes a $j\times j$ matrix with diagonal elements 1 and off-diagonal elements 0.8.}

\end{itemize}
The active set has either cardinality $s_0=|S_0|=3$ or $15$ with $S_0=\{1,2,\cdots,s_0\}$ and the regression coefficients were drawn from a fixed realization of $s_0$ i.i.d. uniform random variables on $[0,2]$. Denote $\mbox{CI}_j$ as a two-sided $95\%$ confidence interval.

In Table~\ref{tab:correct} and \ref{tab:correct2}, we report empirical versions of
\begin{align*}
&\mbox{Avgcov } S_0 = s_0^{-1} \sum_{j\in S_0} \mathbb{P}(\beta_{0j}\in \mbox{CI}_j), &&\mbox{Avglength } S_0 = s_0^{-1} \sum_{j\in S_0} \mbox{length}(\mbox{CI}_j),\\
&\mbox{Avgcov } S_0^c = (p-s_0)^{-1} \sum_{j\in S_0^c} \mathbb{P}(0\in \mbox{CI}_j), &&\mbox{Avglength } S_0^c = (p-s_0)^{-1} \sum_{j\in S_0^c} \mbox{length}( \mbox{CI}_j).
\end{align*}
It is demonstrated that the coverage probabilities are generally close to $95\%$. For active sets with a larger $s_0$, we observe that the confidence intervals are wider, especially for the equal correlation case. This might be because our high dimensional inference procedure is more suitable for very sparse regression. \yu{ When $n>p$, it can be seen that coverage probabilities of confidence intervals based on the de-sparsified Lasso method are mostly closer to the nominal $95\%$ level than those based on partial likelihood method.We further notice from Table~\ref{tab:correct2} that partial likelihood method does not work well for the `equal correlation' case. This indicates that partial likelihood method does not allow strong collinearity among the covariates. }

\begin{table}
\begin{center}
\caption{Average coverage probabilities and lengths of confidence
intervals at the $95\%$ nominal level based on $1000$ repetitions,
where $n = 100$ and $p=500$.} {\footnotesize \label{tab:correct}
    \begin{tabular}{lcccccccc}
    \hline
    & \multicolumn{4}{c}{Active set $S_0=\{1,2,3\}$} & \multicolumn{4}{c}{Active set $S_0=\{1,2,\cdots,15\}$}\\
    \hline
    & \multicolumn{2}{c}{Independent} & \multicolumn{2}{c}{Equal Correlation}& \multicolumn{2}{c}{Independent} & \multicolumn{2}{c}{Equal Correlation}\\
    Censoring Rate & 15\% & 30\% & 15\% & 30\% & 15\% & 30\% & 15\% & 30\%\\
    \hline
    Avgcov $S_0$ & 0.902 &  0.896 & 0.918 & 0.914& 0.930 & 0.952 & 0.915 & 0.903\\
    Avglength $S_0$ & 5.553 & 5.842 & 9.883 & 10.442& 5.467 & 5.982 & 12.477 &14.104\\
    Avgcov $S_0^c$ & 0.967 & 0.965 & 0.943 & 0.934& 0.998 & 0.998 & 0.946 & 0.925\\
    Avglength $S_0^c$ & 4.874 & 5.273 &7.396 & 7.972& 5.176 & 5.675 & 11.264&12.690\\
    \hline
    \end{tabular}
}
\end{center}
\end{table}

\begin{table}
	\begin{center}
		\caption{Average coverage probabilities and lengths of confidence
			intervals at the $95\%$ nominal level based on $1000$ repetitions,
			where $n = 200$ and $p=30$.} {\footnotesize \label{tab:correct2}
			\begin{tabular}{lcccccccc}
				\hline
				\multicolumn{9}{c} {De-sparsified Lasso Method}\\
				\hline
				& \multicolumn{4}{c}{Active set $S_0=\{1,2,3\}$} & \multicolumn{4}{c}{Active set $S_0=\{1,2,\cdots,15\}$}\\
				\hline
				& \multicolumn{2}{c}{Independent} & \multicolumn{2}{c}{Equal Correlation}& \multicolumn{2}{c}{Independent} & \multicolumn{2}{c}{Equal Correlation}\\
				Censoring Rate & 15\% & 30\% & 15\% & 30\% & 15\% & 30\% & 15\% & 30\%\\
				\hline
				Avgcov $S_0$ & 0.837 & 0.867  & 0.921 & 0.925 & 0.833 & 0.836 & 0.878 & 0.879\\
				Avglength $S_0$ & 0.400 & 0.445 & 0.754 & 0.855 & 0.508 & 0.551 & 0.841 & 0.947\\
				Avgcov $S_0^c$ & 0.949 & 0.952 & 0.949 & 0.956 & 0.940 & 0.937 & 0.927 & 0.921\\
				Avglength $S_0^c$ & 0.338 & 0.382 &0.757 & 0.865 & 0.344 & 0.388 & 0.791 & 0.899 \\
				\hline
				\hline
				\multicolumn{9}{c} {Partial Likelihood Method}\\
				\hline
				& \multicolumn{4}{c}{Active set $S_0=\{1,2,3\}$} & \multicolumn{4}{c}{Active set $S_0=\{1,2,\cdots,15\}$}\\
				\hline
				& \multicolumn{2}{c}{Independent} & \multicolumn{2}{c}{Equal Correlation}& \multicolumn{2}{c}{Independent} & \multicolumn{2}{c}{Equal Correlation}\\
				Censoring Rate & 15\% & 30\% & 15\% & 30\% & 15\% & 30\% & 15\% & 30\%\\
				\hline
				Avgcov $S_0$ & 0.823&0.809&0.230&0.401&0.831&0.811&0.229&0.397\\
				Avglength $S_0$ & 0.459&0.513&0.776&0.902&0.459&0.513&0.775&0.902\\
				Avgcov $S_0^c$ & 0.921&0.920&0.927&0.908&0.925&0.919&0.930&0.904\\
				Avglength $S_0^c$ & 0.367&0.414&0.776&0.896&0.367&0.414&0.776&0.895\\
				\hline
				\hline				
			\end{tabular}
		}
	\end{center}
\end{table}

\subsection{misspecified Cox regression model}
In this section, we consider misspecified models. In Tables \ref{tab:mistrue} and \ref{tab:mistrue2}, survival time $T$ was generated from $\lambda(t|{\bm X})=\exp\{X_1^2\}$, and the working model (\ref{cox1}) was used to fit the data in simulations. As explained in Example 1, the pseudo-true parameter $\beta_0=(0,\ldots,0)^T$. we calculated the average coverage probabilities $p^{-1}\sum_{1\le j\le p} \mathbb{P}(\beta_{0j}\in \mbox{CI}_j)$ and average lengths $p^{-1}\sum_{1\le j\le p} \mbox{length}(\mbox{CI}_j)$ by considering two covariance matrices:
\begin{itemize}
    \item Independent: $\Sigma=I$,
    \item Block Equal Correlation I: $\Sigma=\mbox{diag}(1,\widetilde \Sigma_{p-1})$.
\end{itemize}
The asymptotic variance estimates were calculated either from (\ref{var:mis}) (robust) or (\ref{var:correct}) (non-robust). Table~\ref{tab:mistrue} demonstrates that when robust variance estimate is used, the coverage probabilities are closer to the nominal level $95\%$ in comparison with the non-robust formula.

\begin{table}
    \begin{center}
        \caption{Average coverage probabilities and lengths of confidence intervals at the $95\%$ nominal level based on $1000$ repetitions, where $n = 100$ and $p=500$.}
        {\footnotesize \label{tab:mistrue}
            \begin{tabular}{ccccccccc}
                \hline
                & \multicolumn{4}{c}{Independent} & \multicolumn{4}{c}{Block Equal Correlation I}\\
                \hline
                & \multicolumn{2}{c}{15\%} & \multicolumn{2}{c}{30\%} & \multicolumn{2}{c}{15\%} & \multicolumn{2}{c}{30\%}\\
                &Robust & Non-Robust &Robust & Non-Robust &Robust & Non-Robust &Robust & Non-Robust\\
                \hline
                avgcov & 0.930 & 0.900 & 0.937 & 0.904 & 0.946 & 0.938 & 0.952 &0.938\\
                avglength & 1.498 & 1.368 &1.684  & 1.523 & 1.753 & 1.734 & 1.856 & 1.791\\
                \hline

            \end{tabular}
        }
    \end{center}
\end{table}

\begin{table}
	\begin{center}
		\caption{Average coverage probabilities and lengths of confidence intervals at the $95\%$ nominal level based on $1000$ repetitions, where $n = 200$ and $p=30$.}
		{\footnotesize \label{tab:mistrue2}
			\begin{tabular}{ccccccccc}
				\hline
				\multicolumn{9}{c} {De-sparsified Lasso Method}\\
				\hline
				& \multicolumn{4}{c}{Independent} & \multicolumn{4}{c}{Block Equal Correlation I}\\
				\hline
				& \multicolumn{2}{c}{15\%} & \multicolumn{2}{c}{30\%} & \multicolumn{2}{c}{15\%} & \multicolumn{2}{c}{30\%}\\
				&Robust & Non-Robust &Robust & Non-Robust &Robust & Non-Robust &Robust & Non-Robust\\
				\hline
				avgcov & 0.961 & 0.956 & 0.963 & 0.960 & 0.966 & 0.956 & 0.970 &0.961\\
				avglength & 0.360 & 0.343 &0.396  & 0.386 & 0.827 & 0.783 & 0.944 & 0.903\\
				\hline
				\hline
				\multicolumn{9}{c} {Partial Likelihood Method}\\	
				\hline
				& \multicolumn{4}{c}{Independent} & \multicolumn{4}{c}{Block Equal Correlation I}\\
				\hline
				& \multicolumn{2}{c}{15\%} & \multicolumn{2}{c}{30\%} & \multicolumn{2}{c}{15\%} & \multicolumn{2}{c}{30\%}\\
				&Robust & Non-Robust &Robust & Non-Robust &Robust & Non-Robust &Robust & Non-Robust\\
				\hline
				avgcov & 0.913 & 0.919 & 0.922 & 0.919 & 0.915 & 0.921 & 0.920 &0.919\\
				avglength & 0.344 & 0.347 &0.386  & 0.379 & 0.725 & 0.732 & 0.714 & 0.799\\
				\hline			
			\end{tabular}
		}
	\end{center}
\end{table}

Next, we test the null hypothesis $H_0$ that the failure time $T$ does not depend on $X_1$. When the working model (\ref{cox1}) is false, a valid test for $H_0$ based on $\widehat{b}_1$ \yu{and robust variance estimation method} is possible if $\beta_{01}=0$. One example for $\beta_{01}=0$ is that $X_1$ is symmetric about $0$ and independent of other covariates, and $X_1^2$ has an important effect on the true hazard function $\lambda_0(t|{\bm X})$. \yu{Note that the  true  model  need  not  take  an  exponential  regressionform,  and  neither  does  it  have  to  a  proportional  hazards  model.} In Tables \ref{tab:mis} and \ref{tab:mis2}, different true hazards with covariates satisfying these conditions were explored, under the following covariance matrices:
\begin{itemize}
    \item Independent: $\Sigma=I$ ,
    \item Block Equal Correlation II: $\Sigma=\mbox{diag}(1,\tilde{\Sigma}_2,\tilde{\Sigma}_{p-3})$.
\end{itemize}
\yu{Row 1 of Tables \ref{tab:mis} and \ref{tab:mis2} is on the omission of relevant covariates from Cox models, rows 2-3 are on the misspecification of regression forms with possible omission of relevant covariates, and rows 4-5 are on nonproportional hazards models with possible omission of relevant covariates. In particular, row 2 is an additive hazards model and row 4 is an accelerated failure time model.}

Tables \ref{tab:mis} and \ref{tab:mis2} demonstrate that tests based on robust variance estimate give empirical sizes closer to 5\% than those non-robust cases. \yu{This is because test based on robust variance estimation method is asymptotically valid, whereas test based on non-robust variance estimation may not be.} When $n >p$, it is noted that our results based on de-sparsified method are comparable with those based on partial likelihood method.

\begin{table}
\begin{center}
\caption{Empirical Sizes for testing the effect of $X_1$ under the falsely assumed Cox model
    $\lambda(t|{\bm X}) =\lambda(t)\exp({\bm X}^T\beta)$ at the $5\%$ nominal level based on $1000$ repetitions, where $n = 100$ and $p=500$.}
\footnotesize{\label{tab:mis}
    \begin{tabular}{lcccccccc}
    \hline
    True Hazard Function&  \multicolumn{4}{c}{Independent} &  \multicolumn{4}{c}{Block Equal Correlation II}\\
    \hline
    & \multicolumn{2}{c}{15\%} & \multicolumn{2}{c}{30\%}   & \multicolumn{2}{c}{15\%} & \multicolumn{2}{c}{30\%}\\
    &Rob. & Non-R. &Rob. & Non-R. &Rob. & Non-R.&Rob. & Non-R. \\
    \hline
    {1. $\lambda_0(t|{\bm X})=\exp\{X_1^2+.5X_2+X_3\}$} & 0.049  & 0.080  & 0.049 & 0.063& 0.043 & 0.196 & 0.040 & 0.202 \\
    2. $ \lambda_0(t|{\bm X})=\{X_1^2+.5X_2+X_3\}+5$ & 0.049 & 0.083 & 0.044 & 0.074  & 0.056 & 0.126 & 0.049 & 0.110 \\
    3. $\lambda_0(t|{\bm X})=\log(X_1^2+.5X_2+X_3+6)$ & 0.042 & 0.063 & 0.045 & 0.083 & 0.069 & 0.108 & 0.062 & 0.131\\
    4. $\log T= -X_1^2-.5X_2-X_3+\phi$ &0.054 & 0.072 & 0.046 & 0.092 &0.043 & 0.223 & 0.060 & 0.250\\
    5. $T=\exp(-X_1^2-.5X_2-X_3)+\epsilon$& 0.054 & 0.096 & 0.046 & 0.116& 0.049 & 0.158 & 0.048 & 0.180\\
    \hline
	\end{tabular}
}
\end{center}
\end{table}

\begin{table}
	\begin{center}
		\caption{Empirical Sizes for testing the effect of $X_1$ under the falsely assumed Cox model
			$\lambda(t|{\bm X}) =\lambda(t)\exp({\bm X}^T\beta)$ at the $5\%$ nominal level based on $1000$ repetitions, where $n = 200$ and $p=30$.}
		\footnotesize{\label{tab:mis2}
			\begin{tabular}{lcccccccc}
				\hline
				\multicolumn{9}{c} {De-sparsified Lasso Method}\\
				\hline
				True Hazard Function&  \multicolumn{4}{c}{Independent} &  \multicolumn{4}{c}{Block Equal Correlation II}\\
				\hline
				& \multicolumn{2}{c}{15\%} & \multicolumn{2}{c}{30\%}   & \multicolumn{2}{c}{15\%} & \multicolumn{2}{c}{30\%}\\
				&Rob. & Non-R. &Rob. & Non-R. &Rob. & Non-R.&Rob. & Non-R. \\
				\hline
				{1. $\lambda_0(t|{\bm X})=\exp\{X_1^2+.5X_2+X_3\}$} & 0.071  & 0.127  & 0.074 & 0.124& 0.050 & 0.131 & 0.062 & 0.143 \\
				2. $\lambda_0(t|{\bm X})=\{X_1^2+.5X_2+X_3\}+5$ & 0.034 & 0.036 & 0.033 & 0.034  & 0.048 & 0.039 & 0.051 & 0.052 \\
				3. $\lambda_0(t|{\bm X})=\log(X_1^2+.5X_2+X_3+6)$ & 0.047 & 0.049 & 0.041 & 0.039 & 0.036 & 0.035 & 0.049 & 0.046\\
				4. $\log T= -X_1^2-.5X_2-X_3+\phi$ &0.095 & 0.177 & 0.098 & 0.188 &0.056 & 0.188 & 0.051 & 0.184\\
				5. $T=\exp(-X_1^2-.5X_2-X_3)+\epsilon$& 0.058 & 0.085 & 0.064 & 0.084& 0.058 & 0.085 & 0.065 & 0.105\\
				\hline
				\hline
				\multicolumn{9}{c} {Partial Likelihood Method}\\	
				\hline
				& \multicolumn{2}{c}{15\%} & \multicolumn{2}{c}{30\%}   & \multicolumn{2}{c}{15\%} & \multicolumn{2}{c}{30\%}\\
				&Rob. & Non-R. &Rob. & Non-R. &Rob. & Non-R.&Rob. & Non-R. \\
				\hline
				{1. $\lambda_0(t|{\bm X})=\exp\{X_1^2+.5X_2+X_3\}$} & 0.061&0.157&0.075&0.196&0.064&0.191&0.063&0.198 \\
				2. $\lambda_0(t|{\bm X})=\{X_1^2+.5X_2+X_3\}+5$ & 0.069&0.068&0.088&0.100&0.075&0.073&0.065&0.084 \\
				3. $\lambda_0(t|{\bm X})=\log(X_1^2+.5X_2+X_3+6)$ & 0.083&0.084&0.087&0.086&0.085&0.077&0.071&0.079\\
				4. $\log T= -X_1^2-.5X_2-X_3+\phi$ &
				0.058&0.240&0.052&0.259&0.069&0.228&0.060&0.234\\
				5. $T=\exp(-X_1^2-.5X_2-X_3)+\epsilon$&
				0.075&0.137&0.069&0.131&0.101&0.161&0.114&0.170\\
				\hline
			\end{tabular}
		}
	\end{center}
{\footnotesize $\phi$ is a zero-mean normal variable with $0.5$ standard deviation; $\epsilon$ is a standard exponential variable}
\end{table}

\subsection{Real Data Analysis}\label{sec:real}
\yu{We consider a dataset, \cite{alizadeh:2000}: gene-expression data in lymphoma patients. The data (``LymphomaData.rda'') is available in R glmnet package and is publicly available online. The original data is available from {\em http://llmpp.nih.gov/lymphoma/data.shtml.}
There are $n=240$ patients with measurements on $p=7399$ genes. It is of particular interest to find out which genes are important to the disease. We model the data with a high dimensional Cox regression model and obtain the following results for significance. There are $319$ genes out of the total $7399$ genes found significant at individual $5\%$ level based on the robust variance estimation method, while $169$ genes are found significant based on non-robust variance estimation method. This is because the robust variance estimates are generally smaller than the non-robust variance estimates. It is consistent with the findings in \cite{lin:1989} that when the model is correctly specified, robust variance estimates tend to be smaller than non-robust variance estimates, see row 1 of Table 1 in \cite{lin:1989}. This also suggests it may be ideal to model the data with a high dimensional Cox regression model. The Bonferroni-Holm procedure based on $\hat b$ finds no significant coefficient at the $5\%$ significance level for the family-wise error rate (FWER), under either robust or non-robust variance estimate. Similarly as Example 4.3 of \cite{van:2014}, such a low power is expected in presence of thousands of variables.}



\appendix
\section*{Appendix}

\subsection{Notations}

Let $1_n=(1,\dots,1)^T\in \mathbb{R}^n$, $0_p=(0,\dots,0)^T\in\mathbb{R}^p$ and
\begin{eqnarray*}
    &&J=\mbox{diag}\{\underbrace{\widetilde{J},\cdots,\widetilde{J}}_{n}\}, \mbox{~where~} \widetilde J=1_n1_n^T,\\
    &&\widetilde{\Delta}=\mbox{diag}(\underbrace{\Delta_1,\cdots,\Delta_1}_{n},\cdots,\underbrace{\Delta_n,\cdots,\Delta_n}_{n}),\\
    &&B=\mbox{diag}(1(Y_1\geq Y_1),\cdots,1(Y_n\geq Y_1),\cdots,1(Y_1\geq Y_n),\cdots,1(Y_n\geq Y_n)),\\
    &&w_{i,j}(\beta)=\frac{\exp({{\bm X_j}^T\beta})}{\widehat{\mu}_0(Y_i,\beta)},\widetilde{w}_{i,j}(\beta)=\frac{1(Y_j\geq Y_i) \exp({{\bm X_j}^T\beta})/n}{\widehat{\mu}_0(Y_i,\beta)}, \mbox{~for~} i,j=1,\cdots,n,\\
    &&W_{\beta}=\mbox{diag}\left(\sqrt{w_{1,1}(\beta)},\cdots,\sqrt{w_{1,n}(\beta)},\cdots,\sqrt{w_{n,1}(\beta)},\cdots,\sqrt{w_{n,n}(\beta)}\right),\\
    &&\widetilde{W}_{\beta}=\mbox{diag}(\widetilde{w}_{1,1}(\beta),\cdots,\widetilde{w}_{1,n-1}(\beta),1,\cdots,\widetilde{w}_{n,1}(\beta),\cdots,\widetilde{w}_{n,n-1}(\beta),1),\\
    &&\widetilde{X}=({\bm X_1}-{\bm X_n},\cdots,\bm X_{n-1}-{\bm X_n},0_p,\cdots,{\bm X_1}-{\bm X_n},\cdots,\bm X_{n-1}-{\bm X_n},0_p)^T\in \mathbb{R}^{n^2\times p},\\
    &&H_\beta=I-J\widetilde W_\beta, A_{\beta}=W_{\beta}H_\beta\widetilde{X}.
\end{eqnarray*}

We first present some technical lemmas and their proofs.

\subsection{Technical Lemmas}
\begin{lemma}\label{lemma:weight}
	Under Assumptions 1, 2 and 5, 
	we have
	\begin{eqnarray}
	&& JBW_{\beta}^2BJ=nJ,\label{j1}\\
	&&[H_\beta]_{i}^{-1}\nonumber\\
	&&=\frac{1}{\sum_{j=1}^{n-1}\widetilde{w}_{i,j}}\left(
	\begin{array}{ccccc}
	\sum_{j\neq 1, j=1}^{n-1}\widetilde{w}_{i,j} & -\widetilde{w}_{i,2} & \cdots & -\widetilde{w}_{i,n-1} & -1  \\
	-\widetilde{w}_{i,1} & \sum_{j\neq 2, j=1}^{n-1}\widetilde{w}_{i,j} & \cdots & -\widetilde{w}_{i,n-1} & -1  \\
	&  & \cdots &  &    \\
	-\widetilde{w}_{i,1} & -\widetilde{w}_{i,2} & \cdots & \sum_{j\neq n-1, j=1}^{n-1}\widetilde{w}_{i,j} & -1  \\
	-\widetilde{w}_{i,1} & -\widetilde{w}_{i,2} & \cdots & -\widetilde{w}_{i,n-1} & -(1-\sum_{j=1}^{n-1}\widetilde{w}_{i,j})   \\
	\end{array}
	\right)\label{inverIJW},
	\end{eqnarray}
	where $[H_\beta]_i$ is the $i$-th block matrix in $H_\beta$, which is non-singular provided that $Y_n<\max_{i} Y_i$.
	Further, for any $\beta$ between $\beta_0$ and $\widehat{\beta}$,
	\begin{equation}
	\|A_\beta\|_{\infty}=\mathcal{O}_P(1).\label{boundAbeta}
	\end{equation}
\end{lemma}
\begin{proof}
	From the definition, $C_{\beta}=\widetilde{\Delta}BA_{\beta}/n$.
	Note that for any $i=1,\cdots,n$,
	\begin{equation}
	\sum_{j=1}^n \widetilde{w}_{ij}=1.\label{weightsum}
	\end{equation}
	Then (\ref{j1}) and (\ref{inverIJW}) can be obtained from (\ref{weightsum}) and a direct calculation.
	
	By Assumptions 1 - 2, 
	we have $\exp({\bm X}^T\beta_0)$ bounded from above. And from the consistency result for $\widehat{\beta}$ in Lemma 1, 
	we have $\exp({\bm X}^T\beta)$ is bounded above by a constant $U$, for any $\beta$ between $\beta_0$ and $\widehat{\beta}$.
	From Lemma 2 in \cite{kong:2014}, we have with probability tending to 1, for any $t\in [0,\tau]$ and $\beta$ between $\beta_0$ and $\widehat{\beta}$,
	\begin{equation}\label{weightbound}
	1/U^2\leq \exp({\bm X}^T\beta)/\widehat{\mu}_0(t,\beta)\leq 2U^2/\xi,
	\end{equation}
	where $\xi=P(Y\geq \tau)$ as defined in Assumption 5. 
	In view of (\ref{weightbound}), (\ref{boundAbeta}) is true under Assumptions 1, 2 and 4. 
	Direct calculation shows that $\det[H_\beta]=\prod_{i=1}^n \sum_{j=1}^{n-1} \widetilde{w}_{i,j}\neq 0$. Thus $H_\beta$ is non-singular provided that $Y_n<\max_{i} Y_i$.
\end{proof}
\ \\
\indent Lemmas \ref{lemma:gamma1} and \ref{lemma:gamma2} bound the $\ell_1$ and $\ell_2$ differences between $\widehat{\gamma}_{\widehat{\beta},j}$ and $\gamma_{\beta_0,j}$, respectively.

\begin{lemma}\label{lemma:gamma1}
	Under Assumptions 1 - 7, we have 
	\begin{equation*}
	\|\widehat{\gamma}_{\widehat{\beta},j}-\gamma_{\beta_0,j}\|_1=\mathcal{O}_P(s_j\sqrt{\log p/n})+\mathcal{O}_P(\lambda^2s_0/\lambda_j),
	\end{equation*}
	and $$\|\widetilde{\Delta} B A_{\widehat{\beta},-j}(\widehat{\gamma}_{\widehat{\beta},j}-\gamma_{\beta_0,j})\|^2/{n^2}=\mathcal{O}_P(s_j\log p/n)+\mathcal{O}_P(\lambda^2s_0).$$
\end{lemma}

\begin{proof}
	Let $\eta_{\beta_0,j}=A_{\beta_0,j}-A_{\beta_0,-j}\gamma_{\beta_0,j}$. We have
	\begin{eqnarray*}
		A_{\widehat{\beta},j}-A_{\widehat{\beta},-j}\gamma_{\beta_0,j}&&=\{ W_{\widehat{\beta}} H_{\widehat\beta}  H_{\beta_0}^{-1} W_{\beta_0}^{-1}A_{\beta_0}\}_j-\{W_{\widehat{\beta}} H_{\widehat\beta}  H_{\beta_0}^{-1} W_{\beta_0}^{-1}A_{\beta_0}\}_{-j}\gamma_{\beta_0,j}\\
		&&= W_{\widehat{\beta}} H_{\widehat\beta}  H_{\beta_0}^{-1} W_{\beta_0}^{-1}\eta_{\beta_0,j}.
	\end{eqnarray*}
	By the definition of $\widehat{\gamma}_{\widehat{\beta},j}$ and the fact that $C_{\beta}=\widetilde{\Delta} BA_{\beta}/n$, we have, $$\widehat{\gamma}_{\widehat{\beta},j}=\mbox{argmin}_{\gamma} \{\|\widetilde{\Delta} B A_{\widehat{\beta},j}-\widetilde{\Delta} B A_{\widehat{\beta},-j}\gamma\|^2/{n^2}+2\lambda_j\|\gamma\|_1\},$$
	which implies
	\begin{eqnarray*}
		\|\widetilde{\Delta} BA_{\widehat{\beta},j}-\widetilde{\Delta} BA_{\widehat{\beta},-j}\widehat{\gamma}_{\widehat{\beta},j}\|^2/{n^2}+2\lambda_j\|\widehat{\gamma}_{\widehat{\beta},j}\|_1\leq
		\|\widetilde{\Delta} BA_{\widehat{\beta},j}-\widetilde{\Delta} BA_{\widehat{\beta},-j}\gamma_{\beta_0,j}\|^2/{n^2}+2\lambda_j\|\gamma_{\beta_0,j}\|_1.
	\end{eqnarray*}
	Simple algebra shows that
	\begin{eqnarray}\label{eq:uv}
	&&\|\widetilde{\Delta} BA_{\widehat{\beta},-j}(\widehat{\gamma}_{\widehat{\beta},j}-\gamma_{\beta_0,j})\|^2/{n^2}+2\lambda_j\|\widehat{\gamma}_{\widehat{\beta},j}\|_1\notag\\
	&&\leq 2\left(\widetilde{\Delta} BA_{\widehat{\beta},j}-\widetilde{\Delta} BA_{\widehat{\beta},-j}\gamma_{\beta_0,j},\widetilde{\Delta} BA_{\widehat{\beta},-j}(\widehat{\gamma}_{\widehat{\beta},j}-\gamma_{\beta_0,j})\right)/{n^2}+2\lambda_j\|\gamma_{\beta_0,j}\|_1,
	\end{eqnarray}
	where $(u,v):=u^T v$.
	
	Define $$
	(i)=\bigg|\left(\widetilde{\Delta} BA_{\widehat{\beta},j}-\widetilde{\Delta} BA_{\widehat{\beta},-j}\gamma_{\beta_0,j},\widetilde{\Delta} BA_{\widehat{\beta},-j}(\widehat{\gamma}_{\widehat{\beta},j}-\gamma_{\beta_0,j})\right)/{n^2}-(\widetilde{\Delta} B\eta_{\beta_0,j},\widetilde{\Delta} BA_{\widehat{\beta},-j}(\widehat{\gamma}_{\widehat{\beta},j}-\gamma_{\beta_0,j}))/{n^2}\bigg|.
	$$
	We next aim to show
	\begin{align}\label{eq:uv2}
	(i)\le \|\widetilde{\Delta} B A_{\beta_0,-j}(\widehat{\gamma}_{\widehat{\beta},j}-\gamma_{\beta_0,j})\|^2/n^2 + \mathcal{O}_p(\lambda^2s_0).
	\end{align}
	
	Using the Cauchy-Schwartz inequality, we get
	\begin{eqnarray*}
		&&\left|\left(\widetilde{\Delta} BA_{\widehat{\beta},j}-\widetilde{\Delta} BA_{\widehat{\beta},-j}\gamma_{\beta_0,j},\widetilde{\Delta} BA_{\widehat{\beta},-j}(\widehat{\gamma}_{\widehat{\beta},j}-\gamma_{\beta_0,j})\right)-
		\left(\widetilde{\Delta} B\eta_{\beta_0,j},\widetilde{\Delta} BA_{\widehat{\beta},-j}(\widehat{\gamma}_{\widehat{\beta},j}-\gamma_{\beta_0,j})\right)\right|\Big/{n^2}\nonumber\\
		&&=\bigg|\bigg(\widetilde{\Delta} B( W_{\widehat{\beta}} H_{\widehat\beta}  H_{\beta_0}^{-1} W_{\beta_0}^{-1}-I)\eta_{\beta_0,j}, \widetilde{\Delta} BA_{\widehat{\beta},-j}(\widehat{\gamma}_{\widehat{\beta},j}-\gamma_{\beta_0,j})\bigg)\Big/{n^2}\bigg|\nonumber\\
		&&\leq\left\{\|\widetilde{\Delta} B A_{\widehat\beta,-j}(\widehat{\gamma}_{\widehat{\beta},j}-\gamma_{\beta_0,j})\|/n\right\}\left\{\|\widetilde{\Delta} B( W_{\widehat{\beta}} H_{\widehat\beta}  H_{\beta_0}^{-1} W_{\beta_0}^{-1}-I)\eta_{\beta_0,j}\|/n\right\}\\
		&&\le \|\widetilde{\Delta} B A_{\widehat\beta,-j}(\widehat{\gamma}_{\widehat{\beta},j}-\gamma_{\beta_0,j})\|^2/2n^2 + 2 I^2,
	\end{eqnarray*}
	where $I=\|\widetilde{\Delta} B( W_{\widehat{\beta}} H_{\widehat\beta}  H_{\beta_0}^{-1} W_{\beta_0}^{-1}-I)\eta_{\beta_0,j}\|/n$ and the last inequality follows from $2ab \le a^2/2 +2b^2$. To show (\ref{eq:uv2}), now it suffices to show
	\begin{equation}\label{wn2}
	I=\mathcal{O}_p(\lambda\sqrt{s_0}).
	\end{equation}
	
	Define $\widetilde{\eta}:= H_{\beta_0}^{-1} W_{\beta_0}^{-1}\eta_{\beta_0,j}$. We have
	\begin{align*}
	I^2 &= \|\widetilde{\Delta} B( W_{\widehat{\beta}} H_{\widehat\beta} - W_{\beta_0}H_{\beta_0})\widetilde\eta\|^2/n^2\\
	&=\|\widetilde{\Delta} B\{(W_{\widehat{\beta}}-W_{\beta_0})H_{\beta_0}+W_{\beta_0}(H_{\widehat\beta} - H_{\beta_0})+(W_{\widehat{\beta}}-W_{\beta_0})(H_{\widehat\beta} - H_{\beta_0})\}\widetilde\eta\|^2/n^2\\
	&\le \frac{1}{n^2}|\widetilde{\eta}^T H_{\beta_0}^T(W_{\widehat{\beta}}-W_{\beta_0})\widetilde{\Delta} B(W_{\widehat{\beta}}-W_{\beta_0})H_{\beta_0}\widetilde{\eta}|+  \frac{1}{n^2}|\widetilde{\eta}^T J(\widetilde{W}_{\widehat\beta} - \widetilde W_{\beta_0})^TW_{\beta_0}\widetilde{\Delta} BW_{\beta_0}J(\widetilde{W}_{\widehat\beta} - \widetilde W_{\beta_0})\widetilde{\eta}|\\
	&\qquad + \frac{1}{n^2}|\widetilde{\eta}^T J(\widetilde{W}_{\widehat\beta} - \widetilde W_{\beta_0})^T(W_{\widehat{\beta}}-W_{\beta_0})\widetilde{\Delta} B(W_{\widehat{\beta}}-W_{\beta_0})J(\widetilde{W}_{\widehat\beta} - \widetilde W_{\beta_0})\widetilde{\eta}|\\
	&\triangleq (a) + (b)  + (c).
	\end{align*}
	We apply Lemma 6.1 in \cite{van:2014} to bound the RHS of the above inequality. Use (a) as an example, from the mean value theorem, the Cauchy-Schwartz inequality and (\ref{weightbound}),
	\begin{eqnarray*}
		(a)\le 	&&\frac{1}{n^2}\|(W_{\widehat{\beta}}- W_{\beta_0})H_{\beta_0}\widetilde{\eta}\|_2^2=\frac{1}{n^2}\sum_{i=1}^n\sum_{j=1}^n\bigg|\sqrt{\frac{e^{X^T_j\widehat{\beta}}}{\widehat{\mu}_0(Y_i,\widehat{\beta})}}-\sqrt{\frac{e^{{\bm X_j}^T\beta_0}}{\widehat{\mu}_0(Y_i,\beta_0)}}\bigg|^2\|H_{\beta_0}\tilde{\eta}\|_{\infty}\\
		&&\leq \sup_{X,t,\beta}\left(\frac{\exp({\bm X}^T\beta)}{\widehat{\mu}_0(t,\beta)}\right) \frac{1}{2n^2}\sum_{i=1}^n\sum_{j=1}^n|{\bm X_j}^T(\widehat{\beta}-\beta_0)-\{\log\widehat{\mu}_0(Y_i,\widehat{\beta})-\log\widehat{\mu}_0(Y_i,\beta_0)\}|^2\mathcal{O}_P(1)\\
		&&\leq (2U^2/\xi) \frac{1}{n^2}\sum_{i=1}^n\sum_{j=1}^n \left([{\bm X_j}^T(\widehat{\beta}-\beta_0)]^2+\{\log\widehat{\mu}_0(Y_i,\widehat{\beta})-\log\widehat{\mu}_0(Y_i,\beta_0)\}^2\right)\mathcal{O}_P(1)\\
		&&\leq(2U^2/\xi) \bigg\{ \frac{1}{n}\sum_{i=1}^n [{\bm X_i}(\widehat{\beta}-\beta_0)]^2+\frac{1}{n}\sum_{i=1}^n\left(\left[\sup_{t,\beta}\frac{1}{\widehat{\mu}_0(t,\beta)}\right] \{\widehat{\mu}_0 (Y_i,\widehat{\beta})-\widehat{\mu}_0(Y_i,\beta_0)\} \right)^2\bigg\}\mathcal{O}_P(1)\\
		&&\leq(2U^2/\xi) \bigg\{\|{\bm X}^T(\widehat{\beta}-\beta_0)\|^2/n^2+U^2 \frac{1}{n} \sum_{j=1}^n U^2 ({\bm X_j}^T\widehat{\beta}-{\bm X_j}^T\beta_0)^2\bigg\}\mathcal{O}_P(1)\\
		&& =\mathcal{O}_P(\|{\bm X}(\widehat{\beta}-\beta_0)\|^2/n^2)\\
		&&=\mathcal{O}_P(\lambda^2s_0),
	\end{eqnarray*}
	where $\beta$ is between $\beta_0$ and $\widehat{\beta}$ and  the second inequality holds with probability tending to $1$. The last equation is from Lemma 1. 
	Note that Lemma \ref{lemma:weight} and Assumption 3 
	imply $\|\eta_{\beta_0,j}\|_{\infty}\leq \|A_{\beta_0,j}\|_{\infty}+\|A_{\beta_0,-j}\gamma_{\beta_0,j}\|_{\infty}=\mathcal{O}_P(1).$ We have
	$\|H_{\beta_0}\tilde{\eta}\|_{\infty}=\mathcal{O}_P(1)$. 
	Hence, from Lemma 6.1 in \cite{van:2014}, we get $(a)=\mathcal{O}_P(\lambda^2s_0)$. By similar arguments, we can show (\ref{wn2}) holds.
	
	By applying (\ref{eq:uv2}), simple algebra shows that the RHS of (\ref{eq:uv}) can be bounded by
	\begin{eqnarray*}
		\frac{\|\widetilde{\Delta} B A_{\widehat{\beta},j}(\widehat{\gamma}_{\widehat{\beta},j}-\gamma_{\beta_0,j})\|^2}{2n^2}+\frac{2(\widetilde{\Delta} B \eta_{\beta_0,j},\widetilde{\Delta} B A_{\widehat{\beta},-j}(\widehat{\gamma}_{\widehat{\beta},j}-\gamma_{\beta_0,j}))}{n^2}+2\lambda_j\|\gamma_{\beta_0,j}\|_1+\mathcal{O}_P(\lambda^2s_0).
	\end{eqnarray*}
	
	Define the event $\mathcal{E}_n=\{2\|\eta_{\beta_0,j}^TB\tilde{\Delta}A_{\hat\beta,-j}\|_{\infty}/n^2\leq \lambda_j\}$.
	By similar arguments as in linear models \citep{van:2014} and invoking the fact that $\|\widetilde\Delta B A_{\widehat{\beta}}(\widehat{\beta}-\beta_0)\|_2/n^2=\mathcal{O}_P(\|\bf{X}(\widehat{\beta}-\beta_0)\|^2/n)$ \citep{kong:2014}, it can be shown that $P(\mathcal{E}_n)\rightarrow 1.$
	Then we have on the event $\mathcal{E}_n$,
	\begin{eqnarray}
	&&\|\widetilde{\Delta} B A_{\widehat{\beta},-j}(\widehat{\gamma}_{\widehat{\beta},j}-\gamma_{\beta_0,j})\|^2/{2n^2}+2\lambda_j\|\widehat{\gamma}_{\widehat{\beta},j}\|_1\label{equ1:lemmagamma1}\\
	\leq  &&2(\widetilde{\Delta} B \eta_{\beta_0,j},\widetilde{\Delta} B A_{\widehat\beta,-j}(\widehat{\gamma}_{\widehat{\beta},j}-\gamma_{\beta_0,j}))/{n^2}+2\lambda_j\|\gamma_{\beta_0,j}\|_1+\mathcal{O}_P(\lambda^2s_0)\nonumber\\
	\leq  &&\lambda_j \|\widehat{\gamma}_{\widehat{\beta},j}-\gamma_{\beta_0,j}\|_1+2\lambda_j\|\gamma_{\beta_0,j}\|_1+\mathcal{O}_P(\lambda^2s_0). \nonumber
	\end{eqnarray}
	It follows from the triangular inequality that, 
	\begin{align}
	\|\widehat{\gamma}_{\widehat{\beta},j}\|_1=\|\widehat{\gamma}_{\widehat{\beta},j,s_{0j}}\|_1+\|\widehat{\gamma}_{\widehat{\beta},j,s_{0j}^c}\|_1\label{triinequ1}\geq \|\gamma_{\beta_0,j,s_{0j}}\|_1-\|\widehat{\gamma}_{\widehat{\beta},j,s_{0j}}-\gamma_{\beta_0,j,s_{0j}}\|_1+\|\widehat{\gamma}_{\widehat{\beta},j,s_{0j}^c}\|_1,
	\end{align}
	where the subscript $s_{0j}$ denotes the set $\{k: \gamma_{\beta_0,j,k}\neq 0\}.$ 
	Also note that
	\begin{eqnarray}
	\|\widehat{\gamma}_{\widehat{\beta},j}-\gamma_{\beta_0,j}\|_1=\|\widehat{\gamma}_{\widehat{\beta},j,s_{0j}}-
	\gamma_{\beta_0,j,s_{0j}}\|_1+\|\widehat{\gamma}_{\widehat{\beta},j,s_{0j}^c}\|_1.\label{triinequ2}
	\end{eqnarray}
	Plugging (\ref{triinequ1}) into the LHS of (\ref{equ1:lemmagamma1}), and (\ref{triinequ2}) into the RHS of (\ref{equ1:lemmagamma1}), we have
	\begin{eqnarray*}
		&&\|\widetilde{\Delta} B A_{\widehat{\beta},-j}(\widehat{\gamma}_{\widehat{\beta},j}-\gamma_{\beta_0,j})\|^2/{2n^2}+2\lambda_j\|\gamma_{\beta_0,j,s_{0j}}\|_1-2\lambda_j\|\widehat{\gamma}_{\widehat{\beta},j,s_{0j}}-\gamma_{\beta_0,j,s_{0j}}\|_1+2\lambda_j\|\widehat{\gamma}_{\widehat{\beta},j,s_{0j}^c}\|_1\\
		&&\leq \lambda_j \|\widehat{\gamma}_{\widehat{\beta},j,s_{0j}}-\gamma_{\beta_0,j,s_{0j}}\|_1+\lambda_j\|\widehat{\gamma}_{\widehat{\beta},j,s_{0j}^c}\|_1
		+2\lambda_j\|\gamma_{\beta_0,j,s_{0j}}\|_1+\mathcal{O}_P(\lambda^2s_0).
	\end{eqnarray*}
	Therefore, (\ref{equ1:lemmagamma1}) becomes
	\begin{eqnarray*}
		&&\|\widetilde{\Delta} B A_{\widehat{\beta},-j}(\widehat{\gamma}_{\widehat{\beta},j}-\gamma_{\beta_0,j})\|^2/{2n^2}+\lambda_j\|\widehat{\gamma}_{\widehat{\beta},j,s_{0j}^c}\|_1\leq 3\lambda_j\|\widehat{\gamma}_{\widehat{\beta},j,s_{0j}}-\gamma_{\beta_0,j,s_{0j}}\|_1+\mathcal{O}_P(\lambda^2s_0).
	\end{eqnarray*}
	
	Since the smallest eigenvalue of $\Sigma_{\beta_0}$ is bounded away from zero, the compatibility condition holds for $\widehat{\Sigma}$. There exists a constant $\phi_{0j}$, such that
	\begin{eqnarray*}
		&&\quad\|\widetilde{\Delta} B A_{\widehat{\beta},-j}(\widehat{\gamma}_{\widehat{\beta},j}-\gamma_{\beta_0,j})\|^2/{2n^2}+\lambda_j\|\widehat{\gamma}_{\widehat{\beta},j}-\gamma_{\beta_0,j}\|_1\\
		&&=\|\widetilde{\Delta} B A_{\widehat{\beta},-j}(\widehat{\gamma}_{\widehat{\beta},j}-\gamma_{\beta_0,j})\|^2/{2n^2}+
		\lambda_j\|\widehat{\gamma}_{\widehat{\beta},j,s_{0j}}-\gamma_{\beta_0,j,s_{0j}}\|_1+\lambda_j\|\widehat{\gamma}_{\widehat{\beta},j,s_{0j}^c}\|_1\\
		&&\leq 4\lambda_j\|\widehat{\gamma}_{\widehat{\beta},j,s_{0j}}-\gamma_{\beta_0,j,s_{0j}}\|_1+\mathcal{O}_P(\lambda^2s_0)\\
		&&\leq 4\lambda_j [(\widehat{\gamma}_{\widehat{\beta},j}-\gamma_{\beta_0,j})^T\widehat{\Sigma}_{\backslash j,\backslash j}(\widehat{\gamma}_{\widehat{\beta},j}-\gamma_{\beta_0,j})s_j/\phi_{0j}^2]^{1/2}+\mathcal{O}_P(\lambda^2s_0)\\
		&&=4\lambda_j \sqrt{s_j} \|\widetilde{\Delta} B A_{\widehat{\beta},-j}(\widehat{\gamma}_{\widehat{\beta},j}-\gamma_{\beta_0,j})\|_2/(n\phi_{0j})+\mathcal{O}_P(\lambda^2s_0)\\
		&&\leq \|\widetilde{\Delta} B A_{\widehat{\beta},-j}(\widehat{\gamma}_{\widehat{\beta},j}-\gamma_{\beta_0,j})\|^2/{4n^2}+16\lambda_j^2s_j/\phi_{0j}^2+\mathcal{O}_P(\lambda^2s_0).
	\end{eqnarray*}
	The last inequality follows from the basic inequality $4uv\leq u^2/4+16v^2$. Hence,
	\begin{eqnarray*}
		&&\|\widetilde{\Delta} B A_{\widehat{\beta},-j}(\widehat{\gamma}_{\widehat{\beta},j}-\gamma_{\beta_0,j})\|^2/{4n^2}+\lambda_j\|\widehat{\gamma}_{\widehat{\beta},j}-\gamma_{\beta_0,j}\|_1\leq 16 \lambda_j^2 s_j/\phi_{0j}^2+\mathcal{O}_P(\lambda^2s_0).
	\end{eqnarray*}
	Therefore, we deduce that $
	\|\widehat{\gamma}_{\widehat{\beta},j}-\gamma_{\beta_0,j}\|_1=\mathcal{O}_P(\lambda_js_j)+\mathcal{O}_P(\lambda^2s_0/\lambda_j)$ and $    \|\widetilde{\Delta} B A_{\widehat{\beta},-j}(\widehat{\gamma}_{\widehat{\beta},j}-\gamma_{\beta_0,j})\|^2/n^2=\mathcal{O}_P(\lambda_j^2s_j)+\mathcal{O}_P(\lambda^2s_0)$.
\end{proof}
\ \\

\begin{lemma} \label{lemma:gamma2}
	Under Assumptions 1 - 7, we have 
	\begin{equation*}
	\|\widehat{\gamma}_{\widehat{\beta},j}-\gamma_{\beta_0,j}\|=\mathcal{O}_P(\sqrt{s_j\log p/n})+\mathcal{O}_P(\lambda\sqrt{s_0}).
	\end{equation*}
\end{lemma}

\begin{proof}
	Similar to the proof of (\ref{wn2}) and the proof of Theorem 2.4 in \cite{van:2014},
	\begin{eqnarray*}
		&&\bigg|\|\widetilde{\Delta} B A_{\widehat{\beta}}v\|^2/{n^2}-\mathbb{E}\|\widetilde{\Delta} B A_{\beta_0}v\|^2\bigg|\\
		&&\leq \bigg|\|\widetilde{\Delta} B A_{\widehat{\beta}}v\|^2/{n^2}-\|\widetilde{\Delta} B A_{\beta_0}v\|^2/{n^2}\bigg|+\bigg|\|\widetilde{\Delta} B A_{\beta_0}v\|^2/{n^2}-\mathbb{E}\|\widetilde{\Delta} B A_{\beta_0}v\|^2\bigg|\\
		&&=\left\{\|\widetilde{X}^TH_{\widetilde\beta}^TW_{\widehat{\beta}}B\widetilde{\Delta} B W_{\widehat{\beta}} H_{\widehat\beta}\widetilde{X}-\widetilde{X}^TH_{\beta_0}^T W_{\beta_0}B\widetilde{\Delta} B  W_{\beta_0} H_{\beta_0}\widetilde{X}\|_{\infty}+\mathcal{O}_P(\sqrt{\log p/n})\right\}\|v\|_1^2\\
		&&=\left\{\mathcal{O}_P(\lambda^2s_0\sqrt{\log p})+\mathcal{O}_P(\sqrt{\log p/n})\right\}\|v\|_1^2\\
		&&=\mathcal{O}_P(\sqrt{\log p/n})\|v\|_1^2.
	\end{eqnarray*}
	Substituting $v=\widehat{\gamma}_{\widehat{\beta},j}-\gamma_{\beta_0,j}$ into the above inequality, we have
	\begin{eqnarray*}
		&&\quad\|\widetilde{\Delta} B A_{\widehat{\beta},-j}(\widehat{\gamma}_{\widehat{\beta},j}-\gamma_{\beta_0,j})\|^2/{n^2}\\
		&&\geq \mathbb{E}\|\widetilde{\Delta} B A_{\widehat{\beta},-j}(\widehat{\gamma}_{\widehat{\beta},j}-\gamma_{\beta_0,j})\|^2-\mathcal{O}_P(\sqrt{\log p/n})\|\widehat{\gamma}_{\widehat{\beta},j}-\gamma_{\beta_0,j}\|_1^2\\
		&&\geq \Lambda_{\min}^2 \|\widehat{\gamma}_{\widehat{\beta},j}-\gamma_{\beta_0,j}\|^2-\mathcal{O}_P(\sqrt{\log p/n}\times (s_j\sqrt{\log p/n})^2)-\mathcal{O}_P(\sqrt{\log p/n}\times (\lambda^2s_0/\lambda_j)^2)\\
		&&\qquad\qquad-\mathcal{O}_P(\sqrt{\log p/n}\times \lambda^2s_0s_j\sqrt{\log p/n}/\lambda_j),
	\end{eqnarray*}
	where the first inequality follows from triangular inequality, and the second one holds due to Lemma \ref{lemma:gamma1} and the fact that the smallest eigenvalue $\Lambda_{\min}^2$ of $\Sigma_{\beta_0}$ stays away from zero. Again from Lemma \ref{lemma:gamma1}, we note that $\|\widetilde{\Delta} B A_{\widehat{\beta},-j}(\widehat{\gamma}_{\widehat{\beta},j}-\gamma_{\beta_0,j})\|^2/{n^2}=\mathcal{O}_P(s_j\log p/n)+\mathcal{O}_P(\lambda^2s_0)$. Hence, we prove that $ \|\widehat{\gamma}_{\widehat{\beta},j}-\gamma_{\beta_0,j}\|=\mathcal{O}_P(\sqrt{s_j\log p/n})+\mathcal{O}_P(\lambda\sqrt{s_0})$.
\end{proof}

\subsection{Proof of Lemma~\ref{lem:tj1}}
\begin{proof}
We first show that
        \begin{equation}\label{taudif}
        |\widehat{\tau}_{\widehat{\beta},j}^2-\tau_{\beta_0,j}^2|=\mathcal{O}_P(s_j\sqrt{\log p/n})+\mathcal{O}_P(\lambda \sqrt{s_0}).
        \end{equation}
By definition, $\widehat{\tau}_{\widehat{\beta},j}=A_{\widehat{\beta},j}^T B \widetilde{\Delta} (\widetilde{\Delta} B A_{\widehat{\beta},j}-\widetilde{\Delta} B A_{\widehat{\beta},-j}\widehat{\gamma}_{\widehat{\beta},j})/n^2.$ Note that $A_{\widehat{\beta},j}=W_{\widehat{\beta}} H_{\widehat\beta}  H_{\beta_0}^{-1} W_{\beta_0}^{-1}A_{\beta_0,j}$ and $A_{\widehat{\beta},j}-A_{\widehat{\beta},-j}\widehat{\gamma}_{\widehat{\beta},j}= W_{\widehat{\beta}} H_{\widehat\beta}  H_{\beta_0}^{-1} W_{\beta_0}^{-1}\left( A_{\beta_0,j}-A_{\beta_0,-j}\widehat{\gamma}_{\widehat{\beta},j}\right).$ We have
    \begin{align*}
    \widehat{\tau}_{\widehat{\beta},j}^2-\tau_{\beta_0,j}^2&=A_{\beta_0,j}^T  B\widetilde{\Delta} (A_{\beta_0,j}-A_{\beta_0,-j}\widehat{\gamma}_{\widehat{\beta},j})/n^2-\tau_{\beta_0,j}^2\\
    &\quad+A_{\beta_0,j}^T \bigg\{ W_{\beta_0}^{-1}(H_{\beta_0}^T)^{-1} H_{\widehat\beta}^TB \widetilde{\Delta} W_{\widehat{\beta}}^2 \widetilde{\Delta} BH_{\widehat\beta}  H_{\beta_0}^{-1} W_{\beta_0}^{-1}-B\widetilde{\Delta}\bigg\}(A_{\beta_0,j}-A_{\beta_0,-j}\widehat{\gamma}_{\widehat{\beta},j})/n^2\\
    &\triangleq (i) + (ii).
    \end{align*}

Since $A_{\beta_0,j}=\eta_{\beta_0,j}+A_{\beta_0,-j}\gamma_{\beta_0,j}$ and $A_{\beta_0,j}-A_{\beta_0,-j}\widehat{\gamma}_{\widehat{\beta},j}=\eta_{\beta_0,j}-A_{\beta_0,-j}(\widehat{\gamma}_{\widehat{\beta},j}-\gamma_{\beta_0,j})$,
        \begin{eqnarray*}
        	(i)&&\leq|\eta_{\beta_0,j}^T B \widetilde{\Delta} \eta_{\beta_0,j}/n^2-\tau_{\beta_0,j}^2|+|\eta_{\beta_0,j}^T B\widetilde{\Delta} A_{\beta_0,-j}(\widehat{\gamma}_{\widehat{\beta},j}-\gamma_{\beta_0,j})/n^2|\\
        	&&\qquad+|\eta_{\beta_0,j}^T B\widetilde{\Delta} A_{\beta_0,-j}\gamma_{\beta_0,j}/n^2|+{|\gamma_{\beta_0,j}^T A_{\beta_0,-j}^T B\widetilde{\Delta} A_{\beta_0,-j}(\widehat{\gamma}_{\widehat{\beta},j}-\gamma_{\beta_0,j})/n^2|}\\
        	&&\triangleq I + II + III + IV.
        \end{eqnarray*}

For $I$, by the definition of $\tau_{\beta_0,j}$, some simple algebra shows  that $\tau_{\beta_0,j}^2=\mathbb{E}_{X,Y,{\bm X_i},Y_i,\Delta_i}\{(D_{\beta_0,j}-D_{\beta_0,-j}\gamma_{\beta_0,j})^T(D_{\beta_0,j}-D_{\beta_0,-j}\gamma_{\beta_0,j})\}$, where $\mathbb{E}_{X,Y,{\bm X_i},Y_i,\Delta_i}(\cdot)$ is the expectation respect to $(X,Y,{\bm X_i},Y_i,\Delta_i)$.
 Hence,
    \begin{eqnarray*}
        I&&=\bigg[(A_{\beta_0,j}-A_{\beta_0,-j}\gamma_{\beta_0,j})^TB \widetilde{\Delta}(A_{\beta_0,j}-A_{\beta_0,-j}\gamma_{\beta_0,j})/n^2\\
        &&\qquad\qquad-
        \mathbb{E}_{X,Y}\{(D_{\beta_0,j}-D_{\beta_0,-j}\gamma_{\beta_0,j})^T(D_{\beta_0,j}-D_{\beta_0,-j}\gamma_{\beta_0,j})\}\bigg]\\
        &&\qquad+\bigg[\mathbb{E}_{X,Y}\{(D_{\beta_0,j}-D_{\beta_0,-j}\gamma_{\beta_0,j})^T(D_{\beta_0,j}-D_{\beta_0,-j}\gamma_{\beta_0,j})\}\\
        &&\qquad\qquad-\mathbb{E}_{X,Y,{\bm X_i},Y_i,\Delta_i}\{(D_{\beta_0,j}-D_{\beta_0,-j}\gamma_{\beta_0,j})^T(D_{\beta_0,j}-D_{\beta_0,-j}\gamma_{\beta_0,j})\}\bigg].\\
        && \triangleq (a) + (b).
    \end{eqnarray*}

From the equality $\frac{A_1}{B_1}-\frac{A_2}{B_2}=\frac{A_1}{B_1B_2}(B_2-B_1)+\frac{A_1-A_2}{B_2}$ and (\ref{weightbound}), with probability tending to 1,
    \begin{eqnarray*}
    &&A_{\beta_0,j}^T B \widetilde{\Delta} A_{\beta_0,k}/n^2-\mathbb{E}_{X,Y}(D_{\beta_0,j}^T D_{\beta_0,k})\nonumber\\
    &&=\frac{1}{n^2}\sum_{i=1}^n\sum_{l=1}^n\Delta_i \bigg\{\frac{1(Y_l\geq Y_i)\exp(X^T_l\beta_0)}{\widehat{\mu}_0(Y_i;\beta_0)}\bigg[X_{l}^{(j)}-\frac{\widehat{\mu}_{1j}(Y_i;\beta_0)}{\widehat{\mu}_{0}(Y_i;\beta_0)}\bigg]
    \bigg[X_{l}^{(k)}-\frac{\widehat{\mu}_{1k}(Y_i;\beta_0)}{\widehat{\mu}_{0}(Y_i;\beta_0)}\bigg]\nonumber\\
    &&\qquad-\mathbb{E}_{X,Y}\bigg(\frac{1(Y\geq Y_i)\exp({{\bm X}^T\beta_0})}{\mu_0(Y_i;\beta_0)}\bigg[X^{(j)}-\frac{\mu_{1j}(Y_i;\beta_0)}{\mu_{0}(Y_i;\beta_0)}\bigg]
    \bigg[X^{(k)}-\frac{\mu_{1k}(Y_i;\beta_0)}{\mu_{0}(Y_i;\beta_0)}\bigg]\bigg)\bigg\}\nonumber\\
    &&=\frac{1}{n}\sum_{i=1}^n \Delta_i \bigg\{ \bigg[\frac{\widehat{\mu}_{2jk}(Y_i,\beta_0)}{\widehat{\mu}_0(Y_i,\beta_0)}-\frac{\mu_{2jk}(Y_i,\beta_0)}{\mu_0(Y_i,\beta_0)}\bigg]\nonumber
    -\bigg[\frac{\widehat{\mu}_{1j}(Y_i,\beta_0)\widehat{\mu}_{1k}(Y_i,\beta_0)}{\widehat{\mu}_0(Y_i,\beta_0)^2}
    -\frac{\mu_{1j}(Y_i,\beta_0)\mu_{1k}(Y_i,\beta_0)}{\mu_0(Y_i,\beta_0)^2}\bigg]\bigg\}\nonumber\\
    &&\leq 2U/\xi \sup_{t\in[0,\tau]}\bigg|\widehat{\mu}_{2jk}(t,\beta_0)-\mu_{2jk}(t,\beta_0)\bigg|+4U^3K^2/\xi^2 \sup_{t\in[0,\tau]}\bigg|\widehat{\mu}_0(t,\beta_0)-\mu_0(t,\beta_0)\bigg|\nonumber\\
    &&\qquad+4U^2/\xi^2 \sup_{t\in[0,\tau]}\bigg|\widehat{\mu}_{1j}(t,\beta_0)\widehat{\mu}_{1k}(t,\beta_0)-\mu_{1j}(t,\beta_0)\mu_{1k}(t,\beta_0)\bigg|\nonumber\\
    &&\qquad+16U^6K^2/\xi^4\sup_{t\in[0,\tau]}\bigg|\widehat{\mu}_0(t,\beta_0)^2-\mu_0(t,\beta_0)^2\bigg|.
    \end{eqnarray*}
    where $\widehat{\mu}_{1j}$ is the $j$-th element in $\widehat{\mu}_1$, $\mu_{1j}$ is the $j$-th element in $\mu_1$, $\widehat{\mu}_{2jk}$ is the $(j,k)$ element in matrix $\widehat{\mu}_2$, and $\mu_{2jk}$ is the $(j,k)$ element in matrix $\mu_2$. From Lemmas 3 and 4 in \cite{kong:2014}, we obtain $\sup_{t\in[0,\tau]}|\widehat{\mu}_0(t,\beta_0)-\mu_0(t,\beta_0)|=\mathcal{O}_P(\sqrt{\log p/n})$ and $\max_{1\leq j\leq p}\sup_{t\in[0,\tau]}|\widehat{\mu}_{1j}(t,\beta_0)-\mu_{1j}(t,\beta_0)|=\mathcal{O}_P(\sqrt{\log p/n})$.
    We can similarly construct bracketing numbers for the class of functions
    $$
    \mathcal{F}^{jk}=\{1(y\geq t)e^{x\beta}x^{(j)}x^{(k)}/(K^2U):t\in[0,\tau], y\in R, |e^{x\beta}|\leq U, \|x\|_{\infty}\leq K\},$$
    for $j,k=1,\cdots,p.$ Applying Theorem 2.14.9 in \cite{van:1996}, we obtain that
    \begin{eqnarray*}
        \max_{1\leq j\leq p,1\leq k\leq p}\sup_{t\in[0,\tau]}\bigg|\widehat{\mu}_{2jk}(t,\beta_0)-\mu_{2jk}(t,\beta_0)\bigg|=\mathcal{O}_P(\sqrt{\log p/n}).
    \end{eqnarray*}
    Hence,
    \begin{eqnarray}
    &&|A_{\beta_0,j}^T B \widetilde{\Delta} A_{\beta_0,j}/n^2-\mathbb{E}_{X,Y}(D_{\beta_0,j}^T D_{\beta_0,j})|=\mathcal{O}_P(\sqrt{\log p/n}),\label{ADdif1}\\
    &&|\gamma_{\beta_0,j}^TA_{\beta_0,-j}^TB\widetilde{\Delta} A_{\beta_0,j}/n^2-\gamma_{\beta_0,j}^T\mathbb{E}_{X,Y}(D_{\beta_0,-j}D_{\beta_0,j})|\nonumber\\
    &&\qquad\leq
    |\gamma_{\beta_0,j}|_1\|A_{\beta_0,-j}^TB\widetilde{\Delta} A_{\beta_0,j}/n^2-\mathbb{E}_{X,Y}(D_{\beta_0,-j}D_{\beta_0,j})\|_{\infty}\leq \mathcal{O}_P(\sqrt{s_j}\sqrt{\log p/n}),\label{ADdif2}\\
    &&|\gamma_{\beta_0,j}^TA_{\beta_0,-j}^TB\widetilde{\Delta} A_{\beta_0,-j}\gamma_{\beta_0,j}/n^2
    -\gamma_{\beta_0,j}^T\mathbb{E}_{X,Y}(D_{\beta_0,-j}D_{\beta_0,-j})\gamma_{\beta_0,j}|\nonumber\\
    &&\qquad\leq
    |\gamma_{\beta_0,j}|_1^2\|A_{\beta_0,-j}^TB\widetilde{\Delta} A_{\beta_0,-j}/n^2-\mathbb{E}_{X,Y}(D_{\beta_0,-j}D_{\beta_0,-j})\|_{\infty}\leq \mathcal{O}_P(s_j \sqrt{\log p/n}).\label{ADdif3}
    \end{eqnarray}
    Using the fact that $(b)=\mathcal{O}_P(n^{-1/2}),$
    and (\ref{ADdif1}), (\ref{ADdif2}), and (\ref{ADdif3}), we have $I=\mathcal{O}_P(s_j\sqrt{\log p/n})$.

{By substituting $\lambda_j\asymp \sqrt{\log p/n}$ into $\mathcal{E}_n$ in the proof of Lemma~\ref{lemma:gamma1}
, we have $	\|\eta_{\beta_0,j}^TB\tilde{\Delta}A_{\beta_0,-j}\|_{\infty}/n^2=\mathcal{O}_P(\sqrt{\log p/n}).$ }Then together with the inequality $\|\gamma_{\beta_0,j}\|_1\leq \sqrt{s_j}\|\gamma_{\beta_0,j}\|_2=\mathcal{O}_P(\sqrt{s_j})$, we get
\begin{align*}II&=\mathcal{O}_P(\sqrt{\log p/n})\|\widehat{\gamma}_{\widehat{\beta},j}-\gamma_{\beta_0,j}\|_1=\mathcal{O}_P((s_j\vee s_0)\log p/n),\\
III&=\mathcal{O}_P(\sqrt{\log p/n})\|\gamma_{\beta_0,j}\|_1=\mathcal{O}_P(\sqrt{s_j\log p/n}).\end{align*}
For $IV$, it follows from the KKT condition that
    \begin{eqnarray*}
        IV\leq \|A_{\beta_0,-j}^T B\widetilde{\Delta} A_{\beta_0,-j}(\widehat{\gamma}_{\widehat{\beta},j}-\gamma_{\beta_0,j})/n^2\|_{\infty}\|\gamma_{\beta_0,j}\|_1=\mathcal{O}_P(\sqrt{s_j\log p/n}).
    \end{eqnarray*}
    Hence, $(i)=\mathcal{O}_P(s_j\sqrt{\log p/n})$.

    For $(ii)$, $\|A_{\beta_0,j}\|_{\infty}=\mathcal{O}_P(1)$ from (\ref{boundAbeta}) and
$\|A_{\beta_0,-j}\widehat{\gamma}_{\widehat{\beta},j}\|_{\infty}\leq \|A_{\beta_0,-j}\gamma_{\beta_0,j}\|_{\infty}+\mathcal{O}_P(1) \|\widehat{\gamma}_{\widehat{\beta},j}-\gamma_{\beta_0,j}\|_1=\mathcal{O}_P(1).$
    Thus we get
    \begin{equation*}
    (ii)=A_{\beta_0,j}^T W_{\beta_0}^{-1}(H_{\beta_0}^T)^{-1}\{ H_{\widehat\beta}^TB\widetilde{\Delta} W_{\widehat{\beta}}^2 \widetilde{\Delta} B H_{\widehat\beta} -H_{\beta_0}^TB\widetilde{\Delta}  W_{\beta_0}^2\widetilde{\Delta} B H_{\beta_0} \} H_{\beta_0}^{-1} W_{\beta_0}^{-1}
    (A_{\beta_0,j}-A_{\beta_0,-j}\widehat{\gamma}_{\widehat{\beta},j})/n^2.
    \end{equation*}
    Similar to the proof of Theorem 3.2 in \cite{van:2014} and (\ref{wn2}), we have $(ii)=\mathcal{O}_P(\lambda \sqrt{s_0})$. Hence, we have (\ref{taudif}).

    Since $1/\tau_{\beta_0,j}^2=\mathcal{O}(1)$, together with (\ref{taudif}), it implies that
    $$|1/\widehat{\tau}_{\widehat{\beta},j}^2-1/\tau_{\beta_0,j}^2|=\mathcal{O}_P(s_j\sqrt{\log p/n})+\mathcal{O}_P(\lambda \sqrt{s_0}).$$
    We note that
    \begin{eqnarray*}
        &&\|\widehat{\Theta}_j-\Theta_j\|_1=\|\widehat{C}_j/\widehat{\tau}_{\widehat{\beta},j}^2-C_j/\tau_{\beta_0,j}^2\|_1\\
        &&\leq \|\widehat{\gamma}_{\widehat{\beta},j}-\gamma_{\beta_0,j}\|_1/\tau_{\beta_0,j}^2+(\|\gamma_{\beta_0,j}\|_1+1)|1/\widehat{\tau}_{\widehat{\beta},j}^2-1/\tau_{\beta_0,j}^2|\\
        &&= \mathcal{O}_P(s_j\sqrt{\log p/n})+\mathcal{O}_P(\lambda^2s_o/\lambda_j)+\sqrt{s_j}\{\mathcal{O}_P(s_j\sqrt{\log p/n})+\mathcal{O}_P(\lambda \sqrt{s_0})\}\\
        &&=\mathcal{O}_P(\lambda s_j^{3/2}\vee \lambda s_0)=\mathcal{O}_P(1/\sqrt{\log p}).
    \end{eqnarray*}
    Also, we get
    \begin{eqnarray*}
        &&\|\widehat{\Theta}_j-\Theta_j\|_2\leq \|\widehat{\gamma}_{\widehat{\beta},j}-\gamma_{\beta_0,j}\|_2/\tau_{\beta_0,j}^2+(\|\gamma_{\beta_0,j}\|_2+1)|1/\widehat{\tau}_{\widehat{\beta},j}^2-1/\tau_{\beta_0,j}^2|\\
        &&= \mathcal{O}_P(\sqrt{s_j\log p/n})+\mathcal{O}_P(\lambda\sqrt{s_0})+\mathcal{O}_P(s_j\sqrt{\log p/n})+\mathcal{O}_P(\lambda \sqrt{s_0})\\
        &&=\mathcal{O}_P(\lambda \sqrt{s_j}\vee \lambda \sqrt{s_0} \vee \lambda s_j)=\mathcal{O}_P(\lambda \sqrt{s_0} \vee \lambda s_j).
    \end{eqnarray*}
    Moreover, we have
    \begin{eqnarray*}
        &&|\widehat{\Theta}_j^T\Sigma_{\beta_0}\widehat{\Theta}_j-\Theta_{jj}|\leq (\|\Sigma_{\beta_0}\|_{\infty}\|\widehat{\Theta}_j-\Theta_j\|_1^2)\wedge (\Lambda_{\max}^2\|\widehat{\Theta}_j-\Theta_j\|_2^2)+2|1/\widehat{\tau}_{\widehat{\beta},j}^2-1/\tau_{\beta_0,j}^2|.
    \end{eqnarray*}
\end{proof}

\subsection{Proof of Lemma \ref{lemma:scorenormal}}
\begin{proof}
From \cite{lin:1989}, we have $\sqrt{n} \dot{l}_n(\beta_0)\rightarrow^d \mathcal{N}(0,\mathbb{E}\left\{n^{-1}\sum_{i=1}^n v_i(\beta_0)^{\otimes 2}\right\})$. Because $\Theta_j^T \dot{l}_n(\beta_0)$ is bounded by (\ref{weightbound}), we only need to prove that
    \begin{equation}
    \frac{\sqrt{n}\widehat{\Theta}_j \dot{l}_n(\beta_0)}{\widehat{\Theta}_j^T E\left\{n^{-1}\sum_{i=1}^n v_i(\beta_0)^{\otimes 2}\right\}\widehat{\Theta}_j}=\frac{\sqrt{n}{\Theta}_j \dot{l}_n(\beta_0)}{{\Theta}_j^T E\left\{n^{-1}\sum_{i=1}^n v_i(\beta_0)^{\otimes 2}\right\}{\Theta}_j}+o_P(1). \label{scorenormal0}
    \end{equation}
    It follows from Lemma 4 in \cite{kong:2014}, Lemma A.1 in \cite{chernozhukov:2013} and  Lemma \ref{lem:tj1} that
    \begin{eqnarray*}
        &&|(\widehat{\Theta}_j-\Theta_j)^T \dot{l}_n(\beta_0)|\\
        &&\leq \bigg\{\bigg\|n^{-1}\sum_{i=1}^n \bigg\{{\bm X_i}-\frac{\mu_1(Y_i;\beta_0)}{\mu_0(Y_i;\beta_0)}\bigg\}\Delta_i\bigg\|_{\infty}+\sup_{t\in [0,\tau]} \bigg\|\frac{\widehat{\mu}_1(Y_i;\beta_0)}{\widehat{\mu}_0(Y_i;\beta_0)}-\frac{\mu_1(Y_i;\beta_0)}{\mu_0(Y_i;\beta_0)}\bigg\|_{\infty}\bigg\} \|\widehat{\Theta}_j-\Theta_j\|_1\\
        &&=\mathcal{O}_P(\lambda)o_P(1/\sqrt{\log p})=o_P(n^{-1/2}).
    \end{eqnarray*}

    From Lemma 6.1 in \cite{van:2014}, we get
    \begin{eqnarray}
    &&\bigg|\widehat{\Theta}_j^T \mathbb{E}\left\{n^{-1}\sum_{i=1}^n v_i(\beta_0)^{\otimes 2}\right\} \widehat{\Theta}_j-\Theta_j^T \mathbb{E}\left\{n^{-1}\sum_{i=1}^n v_i(\beta_0)^{\otimes 2}\right\} \Theta_j\bigg|\nonumber\\
    &&\leq \bigg\|\mathbb{E}\left\{n^{-1}\sum_{i=1}^n v_i(\beta_0)^{\otimes 2}\right\}\bigg\|_{\infty} \|\widehat{\Theta}_j-\Theta_j\|_1^2+\bigg\|\mathbb{E}\left\{n^{-1}\sum_{i=1}^n v_i(\beta_0)^{\otimes 2}\right\} \Theta_j\bigg\|_{\infty}\|\widehat{\Theta}_j-\Theta_j\|_1\nonumber\\
    &&=o_P(1). \label{scorenormal1}
    \end{eqnarray}
    The last equality holds from Lemma \ref{lem:tj1} and the fact that $\|\mathbb{E}\left\{n^{-1}\sum_{i=1}^n v_i(\beta_0)^{\otimes 2}\right\}\|_{\infty}=\mathcal{O}_P(1)$ and $\|\mathbb{E}\left\{n^{-1}\sum_{i=1}^n v_i(\beta_0)^{\otimes 2}\Theta_j\right\}\|_{\infty}=\mathcal{O}_P(1)$.
    Thus, we obtain (\ref{scorenormal0}).
\end{proof}

\subsection{Proof of Theorem \ref{thm:projection}}
\begin{proof}
It suffices to show that $\beta_0=({\beta}_{01}^*,0)$ is the solution to (\ref{pseudo}). Since $l(\beta)$ is strictly convex around $\beta_0$, we have the solution for
    \begin{equation}
    \dot{l}(\beta)=-\mathbb{E}_{X,Y,\Delta}\bigg\{X\Delta
    -\frac{\mu_1(Y;\beta)}{\mu_0(Y;\beta)}\Delta
    \bigg\}=0 \label{ldot}
    \end{equation}
    is unique, and ${\beta}_{01}^*$ is the unique solution to $\mathbb{E}_{\bm X_1^*,Y,\Delta}\bigg\{\bm X_1^*\Delta
    -\frac{\mu_1^{*}(Y;\beta)}{\mu_0(Y;\beta)}\Delta
    \bigg\}=0$, where $\mu_1^*(Y;\beta)$ is a subvector of $\mu_1(Y;\beta)$ which corresponds to the components of $\bm X_1^*$.
    Define the cumulative conditional hazard function $\Lambda_1(t\mid \bm X)=\int_{0}^t \lambda_0(s\mid \bm X)ds$ and the conditional cumulative distribution function for the censoring time $C$ given ${\bm X}$ as $G_{C\mid \bm X}(c,\bm X)$. We have $\Lambda_1(t\mid \bm X_1^*,X_2^*)=\Lambda_1(t\mid \bm X_1^*)$, $G_{C\mid \bm X}(c,\bm X_1^*,\bm X_2^*)=G_{C\mid \bm X}(c,\bm X_1^*)$ and
    \begin{eqnarray*}
        \mathbb{E}(\bm X_2^*\Delta\mid \bm X_1^*)&&=\mathbb{E}\bigg\{\bm X_2^*\int_0^{\infty} 1-e^{-\Lambda_1(c\mid \bm X)}dG_{C\mid \bm X}(c,\bm X)\mid \bm X_1^*\bigg\}\\
        &&=\mathbb{E}\bigg\{\bm X_2^*\int_0^{\infty} 1-e^{-\Lambda_1(c\mid \bm X_1^*)}dG_{C\mid \bm X}(c,\bm X_1^*)\mid \bm X_1^*\bigg\}=0,\\
    \mu_1(t;(\beta_{01},0))&&=\mathbb{E}_{\bm X_1^*}\bigg( \mathbb{E}_{\bm X_2^*\mid \bm X_1^*}\bigg[ \left(
    \begin{array}{c}
    	\bm X_1^* \\
    	\bm X_2^* \\
    \end{array}
    \right) e^{{\bm X}^T\beta_{0}-\Lambda_1(c\mid \bm X)}\{1-G_{C\mid \bm X}(t,X)\}\bigg|\bm X_1^* \bigg]\bigg)\\&&=\mathbb{E}_{\bm X_1^*}\bigg( \mathbb{E}_{\bm X_2^*\mid \bm X_1^*}\bigg[ \left(
    \begin{array}{c}
    	\bm X_1^* \\
    	\bm X_2^* \\
    \end{array}
    \right) e^{\bm X_1^*\beta_{01}-\Lambda_1(c\mid \bm X_1^*)}\{1-G_{C\mid \bm X}(t,\bm X_1^*)\}\bigg|\bm X_1^* \bigg]\bigg)\\
    &&=\left(
    \begin{array}{c}
    	\mathbb{E}_{\bm X_1^*}\big[\bm X_1^* e^{\bm X_1^*\beta_{01}-\Lambda_1(c\mid \bm X_1^*)}\{1-G_{C\mid \bm X}(t,\bm X_1^*)\}\big] \\
    	0 \\
    \end{array}
    \right).
    \end{eqnarray*}

    The last two equations hold because $\mathbb{E}(\bm X_2^*\mid \bm X_1^*)=0$. Hence, $\beta=({\beta}_{01}^*,0)$ is the solution to (\ref{ldot}). By the uniqueness, we have $\beta_0=({\beta}_{01}^*,0)$.

\end{proof}

\subsection{Proof of Theorem \ref{thm:normalthm}}

\begin{proof}
We first note that
    \begin{eqnarray*}
        \widehat{b}_j-\beta_{0j}&&=\widehat{\beta}_j-\widehat{\Theta}_j^T\dot{l}_n(\widehat{\beta})-\beta_{0j}\\
        &&=\widehat{\beta}_j-\beta_{0j}-\widehat{\Theta}_j^T\dot{l}_n(\beta_0)-\widehat{\Theta}_j^T\ddot{l}_n(\widehat{\beta})(\widehat{\beta}-\beta_0)-Rem_1\\
        &&=-\widehat{\Theta}_j^T\dot{l}_n(\beta_0)-Rem_2,
    \end{eqnarray*} where $Rem_2=(\widehat{\Theta}_j^T\ddot{l}_n(\widehat{\beta})-e_j^T)(\widehat{\beta}-\beta_0)+Rem_1$. And
    \begin{eqnarray*}
    |Rem_1|&&=\bigg|\widehat{\Theta}_j^T\dot{l}_n(\widehat{\beta})-\widehat{\Theta}_j^T\dot{l}_n(\beta_0)-\widehat{\Theta}_j^T\ddot{l}_n(\widehat{\beta})
    (\widehat{\beta}-\beta_0)\bigg|\nonumber\\
    &&=\frac{1}{\widehat{\mu}_0(Y_i;\widehat{\beta})^2 \widehat{\mu}_0(Y_i;\beta_0)}\bigg|\widehat{\Theta}_j^T \widehat{\mu}_1(Y_i;\widehat{\beta})\widehat{\mu}_0(Y_i;\widehat{\beta})\widehat{\mu}_0(Y_i;\beta_0)-\widehat{\Theta}_j^T \widehat{\mu}_1(Y_i;\beta_0)\widehat{\mu}_0(Y_i;\widehat{\beta})^2\nonumber\\
    && \quad -\widehat{\Theta}_j^T \widehat{\mu}_2(Y_i;\widehat{\beta})\widehat{\mu}_0(Y_i;\widehat{\beta})\widehat{\mu}_0(Y_i;\beta_0)(\widehat{\beta}-\beta_0)+\widehat{\mu}_1(Y_i;\widehat{\beta})\widehat{\mu}_0(Y_i;\beta_0)
    \widehat{\mu}_1(Y_i;\widehat{\beta})^T(\widehat{\beta}-\beta_0)\bigg|\nonumber\\
    &&\leq \mathcal{O}_P(1)\frac{1}{n}\sum_{i=1}^n\bigg|\widehat{\Theta}_j^T\bigg\{\widehat{\mu}_1(Y_i;\widehat{\beta})\widehat{\mu}_0(Y_i;\widehat{\beta})\widehat{\mu}_0(Y_i;\beta_0)
    -\widehat{\mu}_1(Y_i;\beta_0)\widehat{\mu}_0(Y_i;\widehat{\beta})\widehat{\mu}_0(Y_i;\beta_0)\nonumber\\
    &&\qquad\qquad \qquad -\widehat{\mu}_2(Y_i;\widehat{\beta})\widehat{\mu}_0(Y_i;\widehat{\beta})\widehat{\mu}_0(Y_i;\beta_0)(\widehat{\beta}-\beta_0)\bigg\}\bigg|\label{rem11}\\
    &&\quad+\bigg|\widehat{\Theta}_j^T\bigg\{\widehat{\mu}_1(Y_i;\beta_0)\widehat{\mu}_0(Y_i;\widehat{\beta})\widehat{\mu}_0(Y_i;\beta_0)-\widehat{\mu}_1(Y_i;\beta_0)\widehat{\mu}_0(Y_i;\widehat{\beta})^2\nonumber\\
    &&\qquad\qquad\qquad +\widehat{\mu}_1(Y_i;\beta_0)\widehat{\mu}_0(Y_i;\widehat{\beta})\widehat{\mu}_1(Y_i;\widehat{\beta})(\widehat{\beta}-\beta_0)\bigg\}\bigg|\label{rem12}\\
    &&\quad +\bigg|\widehat{\Theta}_j^T\bigg\{ \widehat{\mu}_1(Y_i;\widehat{\beta})\widehat{\mu}_0(Y_i;\widehat{\beta})\widehat{\mu}_1(Y_i;\widehat{\beta})(\widehat{\beta}-\beta_0)
    -\widehat{\mu}_1(Y_i;\beta_0)\widehat{\mu}_0(Y_i;\widehat{\beta})\widehat{\mu}_1(Y_i;\widehat{\beta})(\widehat{\beta}-\beta_0)\bigg\}\bigg|\label{rem13}\\
    &&\quad +\bigg|\widehat{\Theta}_j^T\bigg\{ \widehat{\mu}_1(Y_i;\widehat{\beta})\widehat{\mu}_0(Y_i;\beta_0)\widehat{\mu}_1(Y_i;\widehat{\beta})(\widehat{\beta}-\beta_0)
    -\widehat{\mu}_1(Y_i;\widehat{\beta})\widehat{\mu}_0(Y_i;\widehat{\beta})\widehat{\mu}_1(Y_i;\widehat{\beta})(\widehat{\beta}-\beta_0)\bigg\}\bigg|\label{rem14}\\
&&\triangleq  i + ii + iii + iv  .
    \end{eqnarray*}

    Using the mean value theorem, the Cauchy -Schwartz inequality, and the fact that $\|\widehat{\Theta}_j^T X\|_{\infty}=\mathcal{O}_P(1)$, we obtain
    \begin{eqnarray*}
        i&\leq& \mathcal{O}_P(1)\frac{1}{n}\sum_{i=1}^n \bigg|\frac{1}{n}\sum_{k=1}^n 1(Y_k\geq Y_i) [(e^{X_k^T\widehat{\beta}}-e^{X_k^T\beta_0})\widehat{\Theta}_j^T X_k-e^{X_k^T\widehat{\beta}}\widehat{\Theta}_j^T X_k X_k^T (\widehat{\beta}-\beta_0) ]\bigg|\\
        &\leq &\mathcal{O}_P(1)\frac{1}{n}\sum_{k=1}^n \bigg|e^{X_k^T\widehat{\beta}}-e^{X_k^T\beta_0}-e^{X_k^T\widehat{\beta}}X_k^T (\widehat{\beta}-\beta_0)\bigg|\\
        &\leq & \mathcal{O}_P(1)\frac{1}{n}\sum_{k=1}^n \bigg|e^{\tilde{a}_k} X_k(\widehat{\beta}-\beta_0)-e^{X_k^T\widehat{\beta}}X_k^T (\widehat{\beta}-\beta_0)\bigg|\\
        &\leq & \mathcal{O}_P(1)\frac{1}{n}\sum_{k=1}^n |\tilde{a}_k-X_k^T \widehat{\beta}||{\bm X}_k^T (\widehat{\beta}-\beta_0)|\\
        &\leq & \mathcal{O}_P(1) \frac{1}{n}\sum_{k=1}^n [X_k^T(\widehat{\beta}-\beta_0)]^2\\
        &=&\mathcal{O}_P(\lambda^2s_0),
    \end{eqnarray*}
    where $\tilde{a}_k$ is a point intermediating $X_k^T\widehat{\beta}$ and $X_k^T\beta_0$ such that $|\tilde{a}_k-X_k^T\widehat{\beta}|\leq |{\bm X}_k^T (\widehat{\beta}-\beta_0)|$.
    Similarly,
    \begin{eqnarray*}
       ii&\leq& \mathcal{O}_P(1)\frac{1}{n}\sum_{i=1}^n \bigg|\frac{1}{n}\sum_{k=1}^n 1(Y_k\geq Y_i) \{(e^{X_k^T\beta_0}-e^{X_k^T\widehat{\beta}})+e^{X_k^T\widehat{\beta}} X_k^T (\widehat{\beta}-\beta_0) \}\bigg|\\
        &\leq & \mathcal{O}_P(1)\frac{1}{n}\sum_{k=1}^n \bigg|-e^{\tilde{a}_k} X_k(\widehat{\beta}-\beta_0)+e^{X_k^T\widehat{\beta}}X_k^T (\widehat{\beta}-\beta_0)\bigg|\\
        &\leq & \mathcal{O}_P(\lambda^2s_0),\\
        iii&\leq& \mathcal{O}_P(1)\frac{1}{n}\sum_{i=1}^n \bigg|\bigg\{\frac{1}{n}\sum_{k=1}^n 1(Y_k\geq Y_i) (e^{X_k^T\widehat{\beta}}-e^{X_k^T\beta_0})\widehat{\Theta}_j^T X_k\bigg\}\bigg\{\frac{1}{n}\sum_{k=1}^n e^{X_k^T\widehat{\beta}} X_k^T (\widehat{\beta}-\beta_0) \bigg\}\bigg|\\
        &\leq & \mathcal{O}_P(1)\frac{1}{n}\sum_{i=1}^n \bigg|\bigg\{\frac{1}{n}\sum_{k=1}^n 1(Y_k\geq Y_i) e^{\tilde{a}_k} X_k^T (\widehat{\beta}-\beta_0)\bigg\}\bigg\{\frac{1}{n}\sum_{k=1}^n e^{X_k^T\widehat{\beta}} X_k^T (\widehat{\beta}-\beta_0) \bigg\}\bigg|\\
        &\leq & \mathcal{O}_P(1)\bigg\{\frac{1}{n}\sum_{k=1}^n |{\bm X}_k^T (\widehat{\beta}-\beta_0)| \bigg\}^2\\
        &\leq & \mathcal{O}_P(1) \frac{1}{n}\sum_{k=1}^n [X_k^T(\widehat{\beta}-\beta_0)]^2\\
        &=&\mathcal{O}_P(\lambda^2s_0),\\
        iv&\leq& \mathcal{O}_P(1)\frac{1}{n}\sum_{i=1}^n \bigg|\bigg\{\frac{1}{n}\sum_{k=1}^n 1(Y_k\geq Y_i) (e^{X_k^T\widehat{\beta}}-e^{X_k^T\beta_0})\bigg\}\bigg\{\frac{1}{n}\sum_{k=1}^n e^{X_k^T\widehat{\beta}} X_k^T (\widehat{\beta}-\beta_0) \bigg\}\bigg|\\
        &\leq & \mathcal{O}_P(1)\frac{1}{n}\sum_{i=1}^n \bigg|\bigg\{\frac{1}{n}\sum_{k=1}^n 1(Y_k\geq Y_i) e^{\tilde{a}_k} X_k^T (\widehat{\beta}-\beta_0)\bigg\}\bigg\{\frac{1}{n}\sum_{k=1}^n e^{X_k^T\widehat{\beta}} X_k^T (\widehat{\beta}-\beta_0) \bigg\}\bigg|\\
        &\leq & \mathcal{O}_P(\lambda^2s_0).
    \end{eqnarray*}
Hence, $|Rem_1|=\mathcal{O}_P(\lambda^2s_0)$. From  Lemma \ref{lemma:betaconsis},
    \begin{eqnarray*}
        |Rem_2|\leq |Rem_1|+\mathcal{O}_P(\sqrt{\log p/n})\|\widehat{\beta}-\beta_0\|_1=o_P(n^{-1/2}).
    \end{eqnarray*}

    We now show that our variance estimator is consistent. Given that, the asymptotic normality holds from Lemma \ref{lemma:scorenormal}. By the triangular inequality, we have
    \begin{eqnarray*}
        &&\bigg|\widehat{\Theta}_j^T \mathbb{E}\left\{n^{-1}\sum_{i=1}^n v_i(\beta_0)^{\otimes 2}\right\}\widehat{\Theta}_j-\widehat{\Theta}_j^T \left\{n^{-1}\sum_{i=1}^n \widehat{v}_i(\widehat{\beta})^{\otimes 2}\right\}\widehat{\Theta}_j\bigg|\\
        \leq&& \bigg|\widehat{\Theta}_j^T \bigg[n^{-1}\sum_{i=1}^n v_i(\beta_0)^{\otimes 2}-\mathbb{E}\left\{n^{-1}\sum_{i=1}^n v_i(\beta_0)^{\otimes 2}\right\}\bigg]\widehat{\Theta}_j\bigg|+\bigg|\widehat{\Theta}_j^T \bigg[n^{-1}\sum_{i=1}^n \widehat{v}_i(\beta_0)^{\otimes 2}-n^{-1}\sum_{i=1}^n v_i(\beta_0)^{\otimes 2}\bigg]\widehat{\Theta}_j\bigg|\\
        &&+\bigg|\widehat{\Theta}_j^T \bigg[n^{-1}\sum_{i=1}^n \widehat{v}_i(\widehat{\beta})^{\otimes 2}-n^{-1}\sum_{i=1}^n \widehat{v}_i(\beta_0)^{\otimes 2}\bigg]\widehat{\Theta}_j\bigg|\\
      \triangleq && I + II + III.
    \end{eqnarray*}
    Letting $\epsilon_{k,l}:=n^{-1}\sum_{i=1}^n v_i(\beta_0)^{\otimes 2}-\mathbb{E}\left\{n^{-1}\sum_{i=1}^n v_i(\beta_0)^{\otimes 2}\right\}$, we get from Lemma~\ref{lem:tj1} that
    \begin{eqnarray*}
        &&I=|\sum_{k,l}\widehat{\Theta}_{j,k}\widehat{\Theta}_{j,l}\epsilon_{k,l}|\leq \|\widehat{\Theta}_j\|_1^2\|\epsilon\|_{\infty}=\mathcal{O}_P({s_j\lambda})=o_P(1).
    \end{eqnarray*}
    As for $II$, we have
    \begin{eqnarray*}
         II&& = \bigg| n^{-1} \sum_{i=1}^n \left\{\widehat{\Theta}_j^T \widehat{v}_i(\beta_0)+\widehat{\Theta}_j^T v_i(\beta_0)\right\}\left\{\widehat{\Theta}_j^T \widehat{v}_i(\beta_0)-\widehat{\Theta}_j^T v_i(\beta_0)\right\}\bigg|\\
        && \leq \mathcal{O}_P(1) \bigg\{n^{-1} \sum_{i=1}^n \bigg|\widehat{\Theta}_j^T \Delta_i \bigg\{\frac{\widehat{\mu}_1(Y_i;\beta_0)}{\widehat{\mu}_0(Y_i;\beta_0)}-\frac{\mu_1(Y_i;\beta_0)}{\mu_0(Y_i;\beta_0)}\bigg\}\bigg|\\
        &&\quad+\bigg(n^{-1}\sum_{i=1}^n \bigg|n^{-1}\sum_{k=1}^n \Delta_k \bigg[\frac{1(Y_i\geq Y_k)\exp\{{\bm X_i}^T \beta_0\}}{\widehat{\mu}_0(Y_k;\beta_0)}\widehat{\Theta}_j^T\bigg\{{\bm X_i}-\frac{\widehat{\mu}_1(Y_k;\beta_0)}{\widehat{\mu}_0(Y_k;\beta_0)}\bigg\}\\
        && \qquad\qquad -\frac{1(Y_i\geq Y_k) \exp\{{\bm X_i}^T\beta_0\}}{\mu_0(Y_k;\beta_0)}\widehat{\Theta}_j^T \bigg\{{\bm X_i}-\frac{\mu_1(Y_k;\beta_0)}{\mu_0(Y_k;\beta_0)}\bigg\}\bigg]\bigg|\bigg)\\
        &&\quad + n^{-1}\sum_{i=1}^n\bigg|\int_0^{\infty}\frac{1(Y_i\geq t)\exp\{{\bm X_i}^T\beta_0\}}{\mu_0(t;\beta_0)}\widehat{\Theta}_j^T \bigg\{{\bm X_i}-\frac{\mu_1(t;\beta_0)}{\mu_0(t;\beta_0)}\bigg\}d\{\tilde{F}_n(t)-\tilde{F}(t)\}\bigg|\bigg\}\\
        && \triangleq\mathcal{O}_P(1)\{ (a) + (b) +(c)\},
    \end{eqnarray*}
    where $\tilde{F}_n(t)=n^{-1}\sum_{k=1}^n 1(Y_k\leq t, \Delta_k=1)$ and $\tilde{F}(t)=\mathbb{E}F_n(t)$.
    \begin{eqnarray*}
        (a)
        && \leq n^{-1}\sum_{i=1}^n \bigg|\widehat{\Theta}_j^T \Delta_i \bigg\{\frac{\widehat{\mu}_1(Y_i;\beta_0)}{\widehat{\mu}_0(Y_i;\beta_0)\mu_0(Y_i;\beta_0)}\left[\widehat{\mu}_0(Y_i;\beta_0)-\mu_0(Y_i;\beta_0)\right]\bigg\}\bigg|\\
        && \qquad +n^{-1} \sum_{i=1}^n \bigg|\widehat{\Theta}_j^T \Delta_i \frac{\widehat{\mu}_1(Y_i;\beta_0)-\mu_1(Y_i;\beta_0)}{\mu_0(Y_i;\beta_0)}\bigg|\\
        &&{\leq} 4U^2/(n\xi^2) \sum_{i=1}^n \bigg|n^{-1}\sum_{k=1}^n 1(Y_k\geq Y_i) \exp\{X_k^T\beta_0\}\widehat{\Theta}_j^T X_k\bigg|\\
        &&\qquad\bigg|\sup_{t\in[0,\tau]}n^{-1}\sum_{k=1}^n 1(Y_k\geq t) \exp\{X_k^T \beta_0\}-E[1(Y\geq t)\exp\{X^T \beta_0\}]\bigg|\\
        &&\quad +2U \xi \bigg|\sup_{t\in[0,\tau]}n^{-1}\sum_{k=1}^n 1(Y_k\geq t) \exp\{X_k^T \beta_0\}\widehat{\Theta}_j^TX_k-E[1(Y\geq t)\exp\{X^T \beta_0\}\widehat{\Theta}_j^T X]\bigg|\\
        &&\leq 4U^2/\xi^2 U \mathcal{O}_P(1)\mathcal{O}_P(\lambda)+\mathcal{O}_P(1) \|\widehat{\Theta}_j\|_1\mathcal{O}_P(\lambda)=\mathcal{O}_P(\lambda \sqrt{s_j})=o_P(1),
    \end{eqnarray*}
    where the second inequality holds with probability tending to $1$, by Lemma 3 in \cite{kong:2014}, and Assumption \ref{assp:xbeta}. Moreover, we have
    \begin{eqnarray*}
       (b) &&\le -\frac{1(Y_i\geq Y_k) \exp\{{\bm X_i}^T\beta_0\}}{\mu_0(Y_k;\beta_0)}\widehat{\Theta}_j^T \bigg\{{\bm X_i}-\frac{\mu_1(Y_k;\beta_0)}{\mu_0(Y_k;\beta_0)}\bigg\}\bigg]\bigg|\\
        && \leq n^{-1}\sum_{i=1}^n \bigg|n^{-1}\sum_{k=1}^n \Delta_k 1(Y_i\geq Y_k) \exp\{{\bm X_i}^T\beta_0\}\widehat{\Theta}_j^T {\bm X_i} \bigg\{1/\widehat{\mu}_0(Y_k;\beta_0)-1/\mu_0(Y_k;\beta_0)\bigg\}\bigg|\\
        && +n^{-1} \sum_{i=1}^n \bigg|n^{-1}\sum_{k=1}^n \Delta_k 1(Y_i\geq Y_k) \exp\{{\bm X_i}^T\beta_0\} \bigg\{\frac{\widehat{\Theta}_j^T [\widehat{\mu}_1(Y_k;\beta_0)-\mu_1(Y_k;\beta_0)]}{\mu_0(Y_k;\beta_0)^2}\\
        &&\qquad+\frac{\widehat{\Theta}_j^T \widehat{\mu}_1(Y_k;\beta_0)}{\widehat{\mu}_0(Y_k;\beta_0)^2\mu_0(Y_k;\beta_0)^2}[\widehat{\mu}_0(Y_k;\beta_0)+\mu_0(Y_k;\beta_0)][\widehat{\mu}_0(Y_k;\beta_0)-\mu_0(Y_k;\beta_0)]\bigg\}\bigg|\\
        &&\leq U \mathcal{O}_P(1) \sup_{t\in[0,\tau]}\bigg|1/\widehat{\mu}_0(Y_k;\beta_0)-1/\mu_0(Y_k;\beta_0)\bigg|\\
        &&\quad +4U^3/\xi^2 |\widehat{\Theta}_j|_1\sup_{t\in[0,\tau]}\max_{1\leq l\leq p}|\widehat{\mu}_{1l}(t;\beta_0)-\mu_{1l}(t;\beta_0)|+\mathcal{O}_P(1)\sup_{t\in[0,\tau]}|\widehat{\mu}_0(Y_k;\beta_0)-\mu_0(Y_k;\beta_0)|\\
        &&=\mathcal{O}_P(\lambda\sqrt{s_j})=o_P(1).
    \end{eqnarray*}
    As $n^{1/2}\{\bar{F}_n(t)-\bar{F}(t)\}$ converges in distribution to a zero-mean Gaussian process, $(c)=o_P(1)$, which implies that $II=o_P(1)$.

    For $III$, because $\|\Theta_j^TX\|_{\infty}=\mathcal{O}_P(1)$,
    \begin{eqnarray*}
        III &&=\bigg|n^{-1} \sum_{i=1}^n \left\{\widehat{\Theta}_j^T \widehat{v}_i(\widehat{\beta})+\widehat{\Theta}_j^T \widehat{v}_i(\beta_0)\right\}\left\{\widehat{\Theta}_j^T \widehat{v}_i(\widehat{\beta})-\widehat{\Theta}_j^T \widehat{v}_i(\beta_0)\right\}\bigg|\leq \mathcal{O}_P(1) n^{-1} \sum_{i=1}^n |\widehat{\Theta}_j^T \widehat{v}_i(\widehat{\beta})-\widehat{\Theta}_j^T \widehat{v}_i(\beta_0)|\\
        && \leq \mathcal{O}_P(1) \bigg(n^{-1}\sum_{i=1}^n \bigg|\widehat{\Theta}_j^T \bigg[\Delta_i\bigg\{{\bm X_i}-\frac{\widehat{\mu}_1(Y_i;\widehat{\beta})}{\widehat{\mu}_0(Y_i;\widehat{\beta})}\bigg\}
        -\Delta_i\bigg\{{\bm X_i}-\frac{\widehat{\mu}_1(Y_i;\beta_0)}{\widehat{\mu}_0(Y_i;\beta_0)}\bigg\}\bigg]\bigg|\\
        &&\qquad\qquad + n^{-1} \sum_{i=1}^n \bigg|\widehat{\Theta}_j^T \bigg\{n^{-1}\sum_{k=1}^n\frac{\Delta_k 1(Y_i\geq Y_k)\exp\{{\bm X_i}^T \widehat{\beta}\}}{\widehat{\mu}_0(Y_k;\widehat{\beta})}\bigg[{\bm X_i}-\frac{\widehat{\mu}_1(Y_k;\widehat{\beta})}{\widehat{\mu}_0(Y_k;\widehat{\beta})}\bigg]\\
        &&\qquad\qquad -\frac{\Delta_k 1(Y_i\geq Y_k)\exp\{{\bm X_i}^T \beta_0\}}{\widehat{\mu}_0(Y_k;\beta_0)}\bigg[{\bm X_i}-\frac{\widehat{\mu}_1(Y_k;\beta_0)}{\widehat{\mu}_0(Y_k;\beta_0)}\bigg]\bigg\}\bigg|\bigg)\\
        &&\leq \mathcal{O}_P(1) \bigg(n^{-1} \sum_{i=1}^n \bigg|\widehat{\Theta}_j^T \frac{\widehat{\mu}_1(Y_i;\widehat{\beta})}{\widehat{\mu}_0(Y_i;\widehat{\beta})\widehat{\mu}_0(Y_i;\beta_0)}[\widehat{\mu}_0(Y_i;\widehat{\beta})-\widehat{\mu}_0(Y_i;\beta_0)]\bigg|\\
        &&\qquad\qquad + n^{-1} \sum_{i=1}^n \bigg|\widehat{\Theta}_j^T\frac{\widehat{\mu}_1(Y_i;\widehat{\beta})-\widehat{\mu}_1(Y_i;\beta_0)}{\widehat{\mu}_0(Y_i;\beta_0)}\bigg|\\
        &&\qquad\qquad + n^{-1} \sum_{i=1}^n \bigg|\widehat{\Theta}_j^T {\bm X_i}\bigg\{n^{-1}\sum_{k=1}^n \Delta_k 1(Y_i\geq Y_k) \bigg\{\frac{\exp\{{\bm X_i}^T\widehat{\beta}\}}{\widehat{\mu}_0(Y_k;\widehat{\beta})}-\frac{\exp\{{\bm X_i}^T\beta_0\}}{\widehat{\mu}_0(Y_k;\beta_0)}\bigg\}\bigg\}\bigg|\\
        && \qquad\qquad+ n^{-1} \sum_{i=1}^n \bigg|\widehat{\Theta}_j^T n^{-1} \sum_{k=1}^n \Delta_k1(Y_i\geq Y_k)\bigg\{\frac{\widehat{\mu}_1(Y_i;\widehat{\beta})\exp\{{\bm X_i}^T\widehat{\beta}\}}{\widehat{\mu}_0(Y_i;\widehat{\beta})^2}
        -\frac{\widehat{\mu}_1(Y_i;\beta_0)\exp\{{\bm X_i}^T\beta_0\}}{\widehat{\mu}_0(Y_i;\beta_0)^2}\bigg\}\bigg|\bigg)\\
        &&=\mathcal{O}_P(|\widehat{\beta}-\beta_0|_1)=\mathcal{O}_P(\lambda s_0)=o_P(1),
    \end{eqnarray*}
    where the last equality holds from Lemma \ref{lemma:betaconsis}.
\end{proof}

\printendnotes

\bibliography{sample}
\end{document}